\documentclass[amssymb,12pt]{article}

\usepackage{mathtools}
\usepackage{geometry,amssymb,amsthm,amsmath,
graphics,enumerate}
\usepackage{times}
\usepackage{amsfonts}
\usepackage{amscd}
\usepackage{url}
\author{ Michael C.\
Laskowski  
and Danielle S.\ Ulrich
 \thanks{Both authors partially supported
by NSF grant DMS-2154101.}
\\
Department of Mathematics\\University of Maryland
}
\newbox\smilebox
\newbox\anchorbox
\newbox\noanchorbox
\newbox\tempbox

\setbox\smilebox=\hbox{$\smile$}

\def\anchor{\hbox{\vtop{
           \hbox to \wd\smilebox{\hfil\vrule width.4pt height7pt depth1pt\hfil}
           \vskip  -11.5truept
           \hbox to \wd\smilebox{\hfil$\smile$\hfil}}}}
\setbox\anchorbox=\anchor
\def\noanchor{\hbox{\vtop{
           \hbox to \wd\anchorbox{\hfil\anchor\hfil}
           \vskip -14truept
           \hbox to \wd\anchorbox{\hfil/\hfil}}}}
\setbox\noanchorbox=\noanchor

\def\fg#1#2#3{\setbox\tempbox=\hbox{$\scriptstyle{#2}$}
\ifnum\wd\anchorbox>\wd\tempbox\dimen255=\wd\anchorbox
\else\dimen255=\wd\tempbox\fi
{#1\,\vtop{\hbox to \dimen255{\hfil\anchor\hfil}
           \vskip -6truept
           \hbox to \dimen255{\hfil$\scriptstyle{#2}$\hfil}}
           \,#3}}

\def\nfg#1#2#3{\setbox\tempbox=\hbox{$\scriptstyle{#2}$}
\ifnum\wd\noanchorbox>\wd\tempbox\dimen255=\wd\noanchorbox
\else\dimen255=\wd\tempbox\fi
{#1\,\vtop{\hbox to \dimen255{\hfil\noanchor\hfil}
           \vskip -6truept
           \hbox to \dimen255{\hfil$\scriptstyle{#2}$\hfil}}
           \,#3}}

\setbox1=\hbox{$\bot$}

\def\north#1#2{#1\,
\hbox{$\bot$\llap {\hbox to\wd1 {\hfil $/$\hfil}}}
\,#2}

\def\nao#1#2#3{#1\  \hbox{\vtop{ 
\baselineskip=4pt
\hbox{$\bot$\llap {\hbox to\wd1 {\hfil $/$\hfil}}
\hskip .05em \llap{\hbox{$^{\scriptscriptstyle{a}}$}}}\hbox{$\scriptstyle
{#2}$}}}\, #3}

\def\includeE#1{{\lhook\kern-3.5pt\joinrel\smash{
    \mathop{\longrightarrow}\limits^{#1}}}}

\def\efor/{Example~\ref{E4}}

\def\BL/{Baldwin--Lachlan}
\def\Bu/{Buechler}
\def\Hr/{Hrushovski}
\def\lm/{locally modular}
\def\wm/{weakly minimal}
\def\nm/{non--modular}
\def\ss/{superstable}
\def\ud/{unidimensional}
\def\sm/{strongly minimal}

\def\abar{\overline{a}}

\def\bbar{\overline{b}}
\def\cbar{\overline{c}}
\def\dbar{\overline{d}}
\def\ebar{\overline{e}}

\def\hbar{\overline{h}}

\def\rbar{\overline{r}}
\def\sbar{\overline{s}}
\def\tbar{\overline{t}}

\def\xbar{\overline{x}}
\def\ybar{\overline{y}}
\def\zbar{\overline{z}}
\def\Mbar{\overline{M}}

\def\rhat{\hat{r}}

\def\acl{{\rm acl}}

\def\dom{{\rm dom}}

\def\tp{{\rm tp}}
\def\stp{{\rm stp}}

\def\tr/{trivial}
\def\nt/{non--trivial}
\def\st/{strong type}

\def\TV/{Tarski--Vaught}

\def\sc/{sound construction}
\def\ac/{atomic construction}

\def\fal/{functional}
\def\upl/{unique parallel lines}
\def\chp/{categorical in a higher power}

\def\includeE#1{{\lhook\kern-3.5pt\joinrel\smash{
    \mathop{\longrightarrow}\limits^{#1}}}}

\def\efor/{Example~\ref{E4}}

\def\BL/{Baldwin--Lachlan}
\def\Bu/{Buechler}
\def\Hr/{Hrushovski}
\def\lm/{locally modular}
\def\wm/{weakly minimal}
\def\nm/{non--modular}
\def\ss/{superstable}
\def\ud/{unidimensional}
\def\sm/{strongly minimal}

\def\abar{\overline{a}}

\def\bbar{\overline{b}}
\def\cbar{\overline{c}}
\def\dbar{\overline{d}}
\def\ebar{\overline{e}}

\def\hbar{\overline{h}}

\def\rbar{\overline{r}}
\def\sbar{\overline{s}}
\def\tbar{\overline{t}}

\def\xbar{\overline{x}}
\def\ybar{\overline{y}}
\def\zbar{\overline{z}}
\def\Mbar{\overline{M}}

\def\rhat{\hat{r}}

\def\acl{{\rm acl}}

\def\dom{{\rm dom}}

\def\tp{{\rm tp}}
\def\stp{{\rm stp}}

\def\tr/{trivial}
\def\nt/{non--trivial}
\def\st/{strong type}

\def\abar{\bar{a}}
\def\bbar{\bar{b}}
\def\cbar{\bar{c}}
\def\dbar{\bar{d}}
\def\ebar{\bar{e}}
\def\ybar{\bar{y}}
\def\phi{\varphi}

\def\C{{\frak  C}}

\def\F{{\cal F}}
\def\FF{{\bf F}}
\def\M{{\cal M}}

\def\P{{\cal P}}

\def\Nbar{{\overline{N}}}

\def\tp{{\rm tp}}
\def\stp{{\rm stp}}
\def\wt{{\rm wt}}

\def\dom{{\rm dom}}
\def\acl{{\rm acl}}
\def\dcl{{\rm dcl}}

\def\Fa0{{\FF^a_{\aleph_0}}}

\def\<{\langle}
\def\>{\rangle}

\def\V{{\mathbb V}}
\newtheorem{Theorem}{Theorem}[section]
\newtheorem{Proposition}[Theorem]{Proposition}
\newtheorem{Definition}[Theorem]{Definition}

\newtheorem{Remark}[Theorem]{Remark}
\newtheorem{Example}[Theorem]{Example}
\newtheorem{Lemma}[Theorem]{Lemma}
\newtheorem{Corollary}[Theorem]{Corollary}

\newtheorem{Fact}[Theorem]{Fact}
\def\V{{\mathbb V}}

\def\ss{{\bf s}}

\def\sq{{\sqsubseteq}}

\def\PP{{\mathbb P}}
\def\K{{\mathcal K}}

\def\Abar{{\overline{A}}}
\def\Bbar{{\overline{B}}}
\def\Ebar{{\overline{E}}}
\def\Kbar{{\overline{K}}}
\def\PP{{\bf P}}
\def\PPe{{\bf P_e}}

\def\precbf{\preceq_{\infty,\omega}}
\def\F{{\cal F}}
\def\d{{\mathfrak d}}

\def\ehat{\hat{e}}
\def\rhat{\hat{r}}
\def\shat{\hat{s}}
\def\that{\hat{t}}

\newcommand\myrestriction{\mathord\restriction}
\def\mr#1{\myrestriction_{#1}}

\begin{document}

	\title{Equivalents of NOTOP}
	
	\date{\today}

	\maketitle
	
	\begin{abstract}    Working within the context of countable, superstable theories, we give many equivalents of a theory having NOTOP.
	In particular, NOTOP is equivalent to V-DI, the assertion that any type $V$-dominated by an independent triple is isolated over the triple.
	If $T$ has NOTOP, then every model $N$ is atomic over an independent tree of countable, elementary substructures, and hence is
	determined up to back-and-forth equivalence over such  a tree. 
	We also verify Shelah's assertion from Chapter XII of \cite{Shc} that NOTOP implies PMOP (without using NDOP).  
		\end{abstract}

 \section{Introduction}
 The capstone of Shelah's work on classifiable theories was proving the so-called Main Gap.   Shelah found a successive chain of properties of first order theories
 in a countable language, which taken together, characterize which countable theories have fewer than $2^\kappa$ non-isomorphic models of size $\kappa$ for every uncountable cardinal $\kappa$.  The story began with Shelah identifying the notion of a {\em superstable} theory, and proving a myriad of theorems indicating that the class of uncountable models of any non-superstable is extremely unwieldy.   In particular, he proved that unsuperstable theories have the maximal number of non-isomorphic models of every uncountable cardinality, but even more complexity can be found for any of these theories.  With those results in hand, Shelah worked in earnest in understanding models of countable superstable theories.
 
 His next significant advance was the identification of theories with DOP, the Dimensional Order Property.  For this notion, it is cleaner to work in the category of a-models, i.e., models in which every strong type over every finite subset is realized.  On one hand, Shelah proved that superstable theories with NDOP (the negation of DOP) admit a very strong structure theorem on the class of a-models.  In particular, every $a$-model $N$ is a-prime and a-minimal over an independent tree of a-models, each of size at most $2^{\aleph_0}$.
 On the other hand, via the method of quasi-isomorphisms, he proved that any countable theory $T$ that was not superstable with NDOP has the maximum number of models of every uncountable cardinality.  
 
 Whereas this dichotomy is very nice, the grail was to get a decomposition of an arbitrary model $N$ in terms of a tree of countable models (as opposed to a-models).  
 For this, Shelah defined the notion of OTOP, the Omitting Types Order Property, an adjective that could hold for an arbitrary theory. 
 
 \begin{Definition}  {\em  A theory $T$ has {\em OTOP} if there is a type $p(\xbar,\ybar,\zbar)$ such that for every binary relation $R$ on any set $X$, there is a model $M_R$ of $T$ and $\{\abar_i:i\in X\}$ from $M_R$ such that, for every $(i,j)\in X^2$, $M_R$ omits $p(\xbar,\abar_i,\abar_j)$ if and only if $R(i,j)$ holds. \\ $T$ has {\em NOTOP} if it does not have OTOP.
 }
 \end{Definition}

In his book, \cite{Shc}, he explores the consequences of NOTOP among countable, superstable theories with NDOP.  
With it, he was able to solve the  `Main Gap", i.e., to precisely determine the  complete theories in a countable language that have fewer than the maximal number of uncountable models.    
Subsequently, aided by some technical results of Hrushovski,
 Shelah and Buechler succeeded in proving a strong structure theorem for such theories.
   Call a countable, complete theory $T$ {\em classifiable} if it is superstable, with NDOP, and NOTOP.
 In \cite{ShBue} they prove that for a classifiable theory, every model is prime and minimal over (hence uniquely determined by) an independent tree of countable,
 elementary substructures.  
 
 Following this breakthrough, the study of classifiable theories slowed considerably.
   In \cite{HHL}, Hart, Hrushovski, and the first author determined all of the various uncountable spectra of classifiable theories, and some work was done on the `fine structure' of the decompositions as well, e.g., \cite{BBHL}.
   
 Because the property of NOTOP was only discovered significantly after the DOP/NDOP dichotomy was understood, almost nothing was proved for countable, superstable theories with NOTOP, without the additional assumption of NDOP.  
In this paper, we define and develop  numerous technical conditions that are equivalent to NOTOP, the inability of coding arbitrary relations by omitting types
in the context of countable, superstable theories, and we also prove a structure theorem for arbitrary models of countable, superstable theories with NOTOP.

\begin{Theorem}  \label{main}  The following are equivalent for a countable, superstable theory $T$.
\begin{enumerate}
\item  $T$ has V-DI;
\item  $T$ has $\PPe$-NDOP and PMOP;
\item  $T$ has $\PPe$-NDOP and countable PMOP;
\item  $T$ has linear NOTOP;
\item  $T$ has NOTOP.
\end{enumerate}
\end{Theorem}

In particular, we prove that  NOTOP implies PMOP for countable, superstable theories $T$.  This implication had previously been proved  under
the additional assumption of NDOP \cite{Shc,Hart}.
We also show that countable, superstable theories with NOTOP admit a structure theorem for arbitrary models of its theory.  

\begin{Theorem}  \label{decomp}  Suppose $T$ is countable, superstable with NOTOP.   Then  every model $N$
is atomic over an independent tree of countable, elementary substructures.   Thus, $N$ is determined by any such tree, up to
back and forth equivalence
over the tree. 
\end{Theorem}  

\medskip
 
 \centerline{{\bf Throughout the paper, we assume $T$ is a complete, superstable theory in a countable language.}}
 
 \medskip
 
 In Section 2 we define $V$-domination and the property of a theory having V-DI.  It is essentially immediate that V-DI implies countable PMOP, but in
 Section~3 we show that V-DI implies full PMOP, i.e., constructible models exist over independent triples of models of any size. 
 
 In Section~4 we develop an idea
 of Baisalov \cite{Baisalov} by defining a family non-orthogonality classes of regular types $\PPe$ in terms of the existence of a stationary, weight one (not necessarily regular!)
 type in the class being non-isolated over some finite set containing its canonical base.  
Following the thesis of \cite{PNDOP}, we localize the notion of NDOP to the specific class $\PPe$ of regular types and show its relation to V-DI.   With 
Theorem~\ref{sofar} we show that V-DI is equivalent to the conjunction of $\PPe$-NDOP and PMOP.   

On the flip side, we use ideas from \cite{Borel} to study {\em dull pairs} $M\preceq N$, where every $c\in N\setminus M$ has $\tp(c/M)\perp\PPe$.
We see that if $M\preceq N$ is a dull pair, then $M$ and $N$ are back-and-forth equivalent over any finite subset of $M$.  

In Section 6, we show that theories with V-DI admit tree decompositions in the same sense as for classifiable theories.  In both cases,
arbitrarily large models $N$ are atomic over an independent tree of countable, elementary substructures.  In the classifiable case, this is tight, i.e.,
$N$ is prime and minimal over the tree.  Here, we show that if $N,N'$ are two models that admit the same tree decomposition, then $N$ and $N'$ are back-and-forth equivalent
over the tree.

In Section 7, we explore more about the class of regular types $\PPe$.  In the $\omega$-stable case, having $\PPe$-NDOP is equivalent to the older notion of eni-NDOP,
but with Example~\ref{Ex}, we see that they can differ in some countable, superstable theories.

Many of the old, standard results we use are relegated to the Appendix.  There is a small amount of new material in the `na-substructures' subsection, but mostly it is a recording of 
definitions and facts that are presented for the convenience of the reader.  

We are grateful to Saharon Shelah for many insightful conversations about potential variants of OTOP.

\section{$V$-domination and V-DI}

The notion of $V$-domination is due to Harrington, who as early as the 1980's, realized its connection to NOTOP.
Our story starts with an investigation of independent triples of sets.

\begin{Definition} \label{V}  {\em  An {\em independent triple of sets} $\Abar=(A_0,A_1,A_2)$ is any triple of sets satisfying $A_0\subseteq A_1\cap A_2$ and $\fg {A_1} {A_0} {A_2}$.

Given two independent triples $\Abar=(A_0,A_1,A_2)$ and $\Bbar=(B_0,B_1,B_2)$ of sets, we say that {\em $\Bbar$ extends $\Abar$}, written $\Abar\sq\Bbar$ if 
$A_i\subseteq B_i$ for each $i$, $\fg {B_0} {A_0} {A_1A_2}$, $\fg {B_1} {B_0A_1} {A_2}$, and $\fg {B_2} {B_0A_2} {B_1}$ hold. 
}
\end{Definition} 

It is easily checked that the relation $\sq$ is transitive.   Whereas $M_1M_2$ need not be a model for an arbitrary independent triple $\Mbar$,
the category of independent triples with $\sq$ acts similarly to the category of models with $\preceq$.  In particular, we have the following.

\begin{Fact}  \label{category}  
\begin{enumerate}
\item  For any independent triples, $\Mbar\sq\Bbar$ implies $M_1M_2\subseteq_{TV} B_1B_2$ (see Definition~\ref{TVdef});
\item (Upward LS)  For any independent triple $\Abar$, there is an independent triple $\Mbar\sqsupseteq\Abar$ consisting of a-models (or even, $\kappa$-saturated for any cardinal $\kappa$);
\item (Downward LS)  For any independent triple $\Mbar$ of models,  for any infinite cardinal $\lambda$, and for any set $X\subseteq M_1M_2$ with $|X|<\lambda$,
there is $\Abar\sq\Mbar$ with $X\subseteq A_1A_2$ and $|A_1A_2|<\lambda$.
\end{enumerate}
\end{Fact}

\begin{Definition}  \label{Vdom}  {\em 
We say that {\em $c$ is $V$-dominated by $\Abar$} if $\fg c {A_1A_2} {B_1B_2}$ for every $\Bbar\sqsupseteq \Abar$.
}
\end{Definition}

Many facts about $V$-domination are evident.

\begin{Fact}  \label{Vdomfact} 
\begin{enumerate}
\item  If $c$ is $V$-dominated by $\Abar$, then $\stp(c/A_1A_2)\vdash \tp(c/B_1B_2)$ for every $\Bbar\sqsupseteq\Abar$.
\item  If $c$ is $V$-dominated by $\Abar$, then $c$ is $V$-dominated by every $\Bbar\sqsupseteq \Abar$.
\item  If $c$ is $V$ dominated by $\Bbar$ and $\Bbar\sqsupseteq\Abar$ with $\fg {c} {A_1A_2} {B_1B_2}$, then $c$ is $V$-dominated by $\Abar$.
\item  If $\Mbar$ is an independent triple of models and if $\tp(c/M_1M_2)$ is $\ell$-isolated, then $c$ is $V$-dominated by $\Mbar$.
\end{enumerate}
\end{Fact}

\begin{proof}  (1)  Choose $c'$ such that $\stp(c'/A_1A_2)=\stp(c/A_1A_2)$.   Clearly, $c'$ is $V$-dominated by $\Abar$ as well, so we have both
$\fg c {A_1A_2} {B_1B_2}$ and $\fg {c'} {A_1A_2} {B_1B_2}$, hence $\tp(cB_1B_2)=\tp(c'B_1B_2)$.

(2) is immediate by the transitivity of $\sq$.  (3) and (4) are respectively, Lemmas~2.2 and 2.6 of \cite{Hart}.
\qed\end{proof} 

If the independent triple consists of $a$-models, we can say more.

\begin{Fact}   \label{3.3}  Suppose $\Mbar=(M_0,M_1,M_2)$ is an independent triple of $a$-models.  Then the following are equivalent for a finite tuple $c$.
\begin{enumerate}
\item  $c$ is $V$-dominated by $\Mbar$;
\item  There is an independent triple $\Bbar\sq\Mbar$ of finite sets with $\tp(c/B_1B_2)\vdash\tp(c/M_1M_2)$;
\item  $\tp(c/M_1M_2)$ is $a$-isolated;
\item  $\tp(c/M_1M_2)$ is $\ell$-isolated (see Definition~\ref{elldef}).
\end{enumerate}
\end{Fact}

\begin{proof}  $(1)\Rightarrow(2):$  By superstability, choose a finite $X_0\subseteq M_1M_2$ such that $\fg c {X_0} {M_1M_2}$. 
In fact, by either Shelah's Conclusion XII.3.5 in \cite{Shc} or Hart's Relative Stationarity Lemma, there is a finite $X$, $X_0\subseteq X\subseteq M_1M_2$ for which $\tp(c/X)$ is based
and relatively stationary inside $M_1M_2$.   
 Find a finite $\Abar\sq\Mbar$
with $X\subseteq A_1A_2$.  By Fact~\ref{category}(3), $c$ is $V$-dominated by $\Abar$, hence $\stp(c/A_1A_2)\vdash \tp(c/M_1M_2)$ by
Fact~\ref{Vdomfact}(1).  However, by the relative stationarity, this is strengthened to  $\tp(c/A_1A_2)\vdash \tp(c/M_1M_2)$. 

$(2)\Rightarrow(3)$ is trivial, and $(3)\Rightarrow(4)$ is by Conclusion~XII.2.11 of \cite{Shc}.
Finally, $(4)\Rightarrow(1)$ is immediate from Fact~\ref{Vdomfact}(4).
 \qed\end{proof} 

We ostentatiously ask when all of these variants of isolation  over independent triples are actually isolated over the triple.
There are many equivalent ways of formalizing this idea.\footnote{The analogous statement involving  $\ell$-isolation over arbitrary sets
would not be equivalent, as e.g., every type
over a finite set is always $\ell$-isolated.}

\begin{Lemma}  \label{DIequiv}  The following are equivalent for any countable, superstable theory.  
\begin{enumerate}
\item  For every independent triple $\Abar$ of sets and for every finite $c$, if $c$ is $V$-dominated by $\Abar$, then $\tp(c/A_1A_2)$ is isolated;
\item  For every independent triple $\Mbar$ of a-models and for every finite $c$, if $c$ is $V$-dominated by $\Mbar$, then $\tp(c/M_1M_2)$ is isolated;
\item  For every independent triple $\Mbar$ of countable models and for every finite $c$, if $c$ is $V$-dominated by $\Mbar$, then $\tp(c/M_1M_2)$ is isolated;
\item  For every independent triple $\Mbar$ of models and for every finite $c$,  if $\tp(c/M_1M_2)$ is $\ell$-isolated,  then $\tp(c/M_1M_2)$ is isolated;
\item  For every independent triple $\Mbar$ of a-models and for every finite $c$, if $\tp(c/M_1M_2)$ is $\ell$-isolated,  then $\tp(c/M_1M_2)$ is isolated;
\item  For every independent triple $\Mbar$ of countable models and for every finite $c$,  if $\tp(c/M_1M_2)$ is $\ell$-isolated,  then $\tp(c/M_1M_2)$ is isolated.
\end{enumerate}
\end{Lemma}

\begin{proof}   
As (5) and (2) are equivalent by Fact~\ref{3.3}, it suffices to show the equivalence of 
$(1),(2),(3)$ and $(4),(5),(6)$ separately.  
 
 $(1)\Rightarrow(3)$ is trivial.  
$(3)\Rightarrow(2)$: Suppose $c$ is $V$-dominated by a-models $\Nbar$.  Choose a finite $\Bbar\sq\Nbar$ with
$\tp(c/B_1B_2)\vdash \tp(c/N_1N_2)$.   By Fact~\ref{category}(3), choose an independent triple $\Mbar\sq\Nbar$ of countable
models with $B_1B_2\subseteq M_1M_2$.  By Fact~\ref{Vdomfact}(3), $c$ is $V$-dominated by $M_1M_2$, hence $\tp(c/M_1M_2)$ is isolated by
(3).  In particular, $\tp(c/B_1B_2)$ is isolated, hence $\tp(c/N_1N_2)$ is isolated as well.

$(2)\Rightarrow(1)$:   Suppose $c$ is $V$-dominated by $\Abar$.  By Fact~\ref{category}(2), choose a triple $\Mbar$ of a-models with
$\Abar\sq\Mbar$.  Then $c$ is $V$-dominated by $\Mbar$, hence $\tp(c/M_1M_2)$ is isolated.  As $\fg c {A_1A_2} {M_1M_2}$, $\tp(c/A_1A_2)$ is isolated
by the Open Mapping Theorem, Fact~\ref{Open}.  

Turning to $(4),(5),(6)$, $(4)\Rightarrow(6)$ is trivial.  
$(6)\Rightarrow(5)$:  Suppose $\Nbar$ is an independent triple of $a$-models and $\tp(c/N_1N_2)$ is $\ell$-isolated.
As $T$ is countable, choose a countable $A\subseteq N_1N_2$ such that for every $\phi(x,y)$, $\tp_\phi(c/N_1N_2)$ is isolated
by some $L(A)$-formula $\psi(x)$.  By Fact~\ref{category}(3), choose a triple of countable models $\Mbar\sq\Nbar$ with $A\subseteq M_1M_2$.
By (6), $\tp(c/M_1M_2)$ is isolated, hence $\tp(c/N_1N_2)$ is isolated as well by Fact~\ref{category}(1) and Lemma~\ref{TVup}.

$(5)\Rightarrow(4)$:  Say $\Mbar$ is an arbitrary triple of models and $\tp(c/M_1M_2)$ is $\ell$-isolated.  By Fact~\ref{category}(2), choose
an independent triple of a-models $\Nbar\sqsupseteq\Mbar$.   By Fact~\ref{TVupl}, $\tp(c/N_1N_2)$ is also $\ell$-isolated, hence $\tp(c/N_1N_2)$ is isolated by (5).
Thus, $\tp(c/M_1M_2)$ is isolated by the Open Mapping Theorem.  
\qed\end{proof}


\begin{Definition}  {\em  A (countable, superstable) theory $T$ has {\em V-DI,} read {\em V-domination implies isolation}, 
if any one of the conditions in Lemma~\ref{DIequiv} hold.} 
\end{Definition}


%
%

%

Note that if a (countable, superstable) $T$ is V-DI, then a constructible model exists over every independent triple of {\em countable} models.  
[Given such an $\Mbar$, by Fact~\ref{ellexist}(2) choose a countable, $\ell$-atomic model over $M_1M_2$.  Since $T$ is V-DI, $N$ is atomic over $M_1M_2$, hence
constructible by Fact~\ref{arbitrary}(2).]
In the next section we will improve this  by showing  that V-DI implies the existence of a  constructible model  over any independent triple of models.

\section{V-DI implies PMOP}

\begin{Definition}  {\em  A theory $T$ has {\em Prime Models Over Pairs,} PMOP, if there is a constructible model over every independent triple of models.
}
\end{Definition}

The main goal of this section is to prove Theorem~\ref{DIPMOP}, that a countable superstable theory $T$ with V-DI also has PMOP.
Although V-DI was never explicitly described, this result was essentially proved by both Shelah \cite{Shc} and Hart \cite{Hart}, under the additional assumption of
NDOP.  Here, we prove the theorem without assuming NDOP.   Curiously, the proof without NDOP is arguably more straightforward than the either of the previous proofs.

For this, we  pass from independent triples of models to certain {\em stable systems} of models.  All of this development is due to Shelah and can be found in
Chapter XII of \cite{Shc}.

\begin{Definition} \label{stablesystem}   {\em  For $n\ge 2$, let $\P^-(n)$ denote the partial order defined by subset on the set $\P(n)\setminus\{n\}$ with $2^n-1$ elements.
A {\em $\P^-(n)$-stable system of models $\Mbar=(M_s:s\in \P^-(n))$} satisfies:
\begin{itemize}
\item  $M_t\preceq M_s$ whenever $t\subseteq s$; and
\item for each $s$, $\fg {M_s} {\Mbar_{<s}} {\bigcup\{M_t:t\not\supseteq s\}}$.
\end{itemize}
Given two $\P^-(n)$-stable systems $\Mbar,\Nbar$, we say $\Mbar\sq\Nbar$ if, for each $s\in \P^-(n)$, $M_s\preceq N_s$ and
$\fg {N_s}  {M_s\bigcup\Nbar_{\subset s}}  {\bigcup\{N_t:t\not\supseteq s\}}$.

We say a finite tuple $c$ is {\em $\P^-(n)$-dominated} by  $\P^-(n)$-stable system $\Mbar$ if 
$\fg c {\bigcup\Mbar} {\bigcup \Nbar}$ for all $\Nbar\sqsupseteq\Mbar$.
}
\end{Definition}

Note that an independent triple $(M_0,M_1,M_2)$ of models is precisely a $\P^-(2)$-stable system of models and the definitions of $\sq$ given in definitions \ref{V} and \ref{stablesystem} coincide.  As well, Facts~\ref{category}, \ref{Vdomfact} and  \ref{3.3} go through in this more general setting.

\begin{Fact}   \label{3.3+}  Suppose $\Mbar$ is a $\P^-(n)$-stable system of $a$-models. Then the following are equivalent for a finite tuple $c$.
\begin{enumerate}
\item  $c$ is $\P^-(n)$-dominated by $\Mbar$;
\item  There is some $\P^-(n)$-system  $\Bbar\sq\Mbar$ of finite sets such that $\tp(c/\bigcup\Bbar)\vdash\tp(c/\bigcup\Mbar)$;
\item  $\tp(c/\bigcup\Mbar)$ is $a$-isolated;
\item  $\tp(c/\bigcup\Mbar)$ is $\ell$-isolated.
\end{enumerate}
\end{Fact}

The following notion is a simplification of Hart's `$\lambda$-special $\P^-(n)$-system' and Shelah's sp.\ stable system in XII.5.1 of \cite{Shc}.
The crucial distinction is that we only require our "special" nodes to be atomic, as opposed to constructible.  By e.g.,  Fact~\ref{arbitrary}(2), the notions coincide
over systems of countable models, but typically are distinct over uncountable systems.

\begin{Definition}  {\em  An {\em atomic-special $\P^-(n)$-stable system of models} has the additional property that for all $s\in \P^-(n)$ with $\{0,1\}\subseteq s$,
$M_s$ is atomic over $\bigcup\Mbar_{<s}$.\\
We say $T$ has {\em a-s $\P^-(n)$-DI} if there is an atomic model $M^*$ over $\bigcup\Mbar$ for every atomic-special $\P^-(n)$-stable system of models.
}
\end{Definition}

Note that any independent triple $\Mbar=(M_0,M_1,M_2)$ of models is an atomic-special $\P^-(2)$-stable system of models, so V-DI implies
a-s $\P^-(2)$-DI by the discussion at the end of Section~2.    The following Proposition extends this to higher dimensional a-s systems.

\begin{Proposition}  \label{a-s}  Suppose $T$ has V-DI.  Then $T$ has  a-s $\P^-(n)$-DI for all $n\ge 2$.
\end{Proposition}

\begin{proof} We prove this by induction on $n\ge 2$.  That $T$ has a-s $P^-(2)$-DI was noted above, so fix $n\ge 2$ and assume $T$ has  a-s $\P^-(k)$-DI for all $2\le k\le n$.
Let $\Mbar=(M_s:s\in\P^-(n+1))$ be an atomic-special $\P^-(n+1)$-stable system.   We first note that, under these hypotheses,  the standard method of blowing up 
$\Mbar$  to a $\P^-(n+1)$-system of $\aleph_1$-saturated models preserves being atomic-special.   To see that, choose an enumeration $(s_i:i<2^{n+1}-1)$ of $\P^-(n+1)$
such that $s_i\subseteq s_j$ implies $i\le j$.  We recursively construct a sequence $(N_{s_i})$ of $\aleph_1$-saturated models satisfying
$\fg {N_{s_i}} {M_{s_i}\bigcup\{N_{s_j}:j<i\}} {\bigcup\{M_t:t\not\supseteq s\}}$ and when $|s_i|\ge 2$, $N_{s_i}$ is chosen to be $\aleph_1$-prime over
$M_{s_i}\bigcup\Nbar_{<s_i}$.   It follows from XII.2.6 of \cite{Shc} that the resulting  system $\Nbar=(N_s:s\in\P^-(n+1))$ is a stable system, and in fact $\Mbar\sq\Nbar$.
To see that $\Nbar$ is atomic-special, we argue by induction on $i$, that if 
$\{0,1\}\subseteq s_i$, then $N_{s_i}$ is atomic over $\Nbar_{<s_i}$. Choose $i$ for which $\{0,1\}\subseteq s_i$ and assume that this holds for all $j<i$.  Let $k=|s_i|$.
Our inductive hypothesis implies the subsystem $\Nbar_{\subset s_i}$ is an atomic-special $\P^-(k)$-system.
   Since $\Mbar$ was assumed to be atomic-special, $M_{s_i}$ is atomic over
$\Mbar_{\subset s_i}$.  However, since $\bigcup\Mbar_{\subset s_i}\sq \bigcup\Nbar_{\subset s_i}$, we also have that 
the set $M_{s_i}$ is atomic over $\bigcup\Nbar_{\subset s_i}$ by Lemma~\ref{TVup}.
As $N_{s_i}$ was chosen to be $\aleph_1$-atomic over $M_{s_i}\Nbar_{\subset s_i}$, it follows that $N_{s_i}$ is also $\aleph_1$-atomic over $\Nbar_{\subset s_i}$.  
Since  the subsequence $(N_t:t\subset s_i)$  is an a-s  $\P^-(k)$-stable system of $\aleph_1$-saturated models, it follows from 
Fact~\ref{3.3+} that $N_{s_i}$ is $\P^-(k)$-dominated by $(N_t:t\subset s_i)$.  Thus, $N_{s_i}$ is atomic over $\Nbar_{\subset s_i}$ since  a-s $\P^-(k)$-DI holds.

Now suppose $c$ is $\P^-(n+1)$-dominated by $\Mbar$ and choose  an a-s $\P^-(n+1)$-stable system $\Nbar\sqsupseteq\Mbar$ consisting of $\aleph_1$-saturated models.
We will show that $\tp(c/\bigcup\Nbar)$ is isolated, which, since $\fg c {\Mbar} {\Nbar}$, directly implies $\tp(c/\Mbar)$ is isolated by the Open Mapping Theorem.  

Next, we unpack $\Nbar$ into three pieces.  Let $\K_0=(N_s:s\in \P^-(n))$ and let $\K_1=(N_{s\cup\{n\}}:s\in \P^-(n))$, with the third, remaining piece $N_n$.
Note that $\K_0\sq\K_1$ as $\P^-(n)$-systems; $\bigcup \K_0\subseteq \bigcup \K_1$, hence $\bigcup\Nbar=\bigcup\K_1\cup\{N_n\}$; and that $\K_1$ is an a-s $\P^-(n)$-stable system.  

\medskip\noindent{\bf Claim.}  $N_n c$ is $\P^-(n)$-dominated by $\K_1$.

\begin{proof}   Since $N_n$ is atomic over $\bigcup \K_0$, $N_n$ is $\P^-(n)$-dominated by $\K_0$.   As $\K_0\sq \K_1$, it follows that
$N_n$ is $\P^-(n)$-dominated by $\K_1$ as well.     Now choose any $\P^-(n)$-stable system $\Ebar\sqsupseteq\K_1$.  It follows that
$$\fg {N_n} {\bigcup\K_1} {\bigcup\Ebar}$$
Form a $\P^-(n+1)$-stable system $\Ebar^*$ by piecing together $\K_0$, $\Ebar$, and $N_n$.  It is readily checked that $\Nbar\sq\Ebar^*$ as $\P^-(n+1)$-structures,
hence $\fg c {\bigcup\Nbar} {\bigcup \Ebar^*}$, so by transitivity we have $\fg {cN_n} {\bigcup\K_1} {\bigcup\Ebar}$, proving the Claim.
\qed\end{proof}

\smallskip

By the Claim and $\K_1$ being an a-s $\P^-(n)$-DI, we have that $cN_n$ is atomic over $\bigcup\Kbar_1$.   Additionally, by Fact~\ref{3.3},
there is a finite $b\subseteq \bigcup\Nbar$ for which
$$\tp(c/b)\vdash\tp(c/\bigcup\Nbar)$$
Recall that $\bigcup\Nbar=\bigcup\K_1\cup N_n$.  
Thus, $\tp(bc/\Kbar_1)$ is isolated, which implies $\tp(c/\Kbar_1 b)$ is isolated as well.  As $\tp(c/b)\vdash\tp(c/\Nbar)$, the same formula isolates
$\tp(c/\Nbar)$ as well.  
\qed\end{proof}

At this point, we could simply quote Theorem 3.3 of \cite{Hart} to conclude the following theorem.  However, our new notion of being atomically special simplifies the argument somewhat.

\begin{Theorem}  \label{DIPMOP}  If a countable, superstable theory $T$ has V-DI, then it has PMOP.
\end{Theorem}

\begin{proof}  We argue that for all infinite cardinals $\kappa$, 
\begin{quotation}  
\noindent  For all $n\ge 2$, there is a constructible model $N$ over any atomic-special $\P^-(n)$-system $\Mbar$ with $|\bigcup \Mbar|\le \kappa$. \hskip3in $(**)$
\end{quotation}
To begin, note that this holds for $\kappa=\aleph_0$  by coupling Fact~\ref{arbitrary}(2) with  Proposition~\ref{a-s}.
So fix an uncountable cardinal $\kappa$ and assume we have the above for all $\lambda<\kappa$.  Choose $n\ge 2$ and an atomic-special $\P^-(n)$-stable system
$\Mbar=(M_s:s\in\P^-(n))$ with $|\bigcup\Mbar|=\kappa$.    Let $\mu=cf(\kappa)$.
By iterating the analogue of Fact~\ref{category}(3),  choose an elementary chain $(\Mbar^\alpha:\alpha<\mu)$ of $\P^-(n)$-systems such that 
\begin{enumerate}
\item  $|\bigcup\Mbar^\alpha|<\kappa$;
\item $\Mbar^\alpha\sq \Mbar^{\alpha+1}$ for all $\alpha$; and
\item  $\Mbar=\bigcup\{\Mbar^\alpha:\alpha<\mu\}$.
\end{enumerate}
We will recursively build a sequence of sequences $\cbar_\alpha$ for each $\alpha$ such that
\begin{enumerate}
\item  $\cbar_\alpha$ is an initial segment of $\cbar_\beta$ whenever $\alpha\le\beta<\mu$:
\item  $\cbar_\alpha$ enumerates a constructible model $N_\alpha$ over $\bigcup \Mbar^\alpha$;
\item  $\cbar_\alpha$ is also a construction sequence over $\Mbar$ (hence also over any $\Mbar^\beta$, $\beta\ge\alpha$).
\end{enumerate}
To begin, since $|\bigcup \Mbar^0|<\kappa$, apply $(**)$ to get $\cbar_0$, enumerating a constructible model over $\Mbar^0$.
As $\bigcup\Mbar^0\subseteq_{TV} \Mbar$, $\cbar_0$ is also a construction sequence over $\bigcup\Mbar$.   For $\gamma<\mu$ a non-zero limit, 
take $\cbar_\gamma$ to be the concatenation of all $\cbar_\alpha$, $\alpha<\gamma$.

Say $\alpha<\mu$ and $\cbar_\alpha$ has been found.  Let $N_\alpha=\bigcup\cbar_\alpha$.
Note that since $\Mbar^\alpha\sq\Mbar^{\alpha+1}$, the three pieces $(\Mbar^\alpha,\Mbar^{\alpha+1},N_\alpha)$ form a $\P^-(n+1)$-stable system.  
We also claim that it is atomic-special.  For this, choose $s\in \P^-(n+1)$ with $\{0,1\}\subseteq s$.  There are three cases.
First, if $s=n=\{i:i\in n\}$, then  as $N_\alpha$ is constructible over $\bigcup\Mbar^\alpha$, it is atomic over $\bigcup\Mbar^\alpha$ as well.
Second, if $s\subset n$, $s\neq n$, then by $\Mbar^\alpha\sq\Mbar$ we have $\fg {M_s^\alpha} {\M_{<s}^\alpha}  {\Mbar_{<s}}$.
But, as $\Mbar$ is atomic-special, every finite $e\in M_s^\alpha$ has $\tp(e/\bigcup \Mbar_{<s})$ isolated.  Thus, by the Open Mapping Theorem,
$\tp(e/\Mbar^{\alpha+1}_{<s})$ is isolated as well.  
Third, if $n\in s$, then similarly, $\fg {M_s^{\alpha+1}} {\Mbar_{<s}^{\alpha+1}}  {\Mbar_{<s}}$ and
$\tp(e/\bigcup \Mbar_{<s})$ is isolated for all finite $e\in M_s^{\alpha+1}$, so again, $\tp(e/\Mbar^{\alpha+1}_{<s})$ is isolated by the Open Mapping Theorem.

Thus, by our inductive hypothesis, there is a constructible model $N_{\alpha+1}$ over $\bigcup\Mbar^{\alpha+1}\cbar_\alpha$.  However, as $\Mbar^\alpha\subseteq_{TV}
\Mbar^{\alpha+1}$, $\cbar_\alpha$ is also constructible over $\bigcup\Mbar^{\alpha+1}$.  Thus, there is a construction sequence $\cbar_{\alpha+1}$ end extending $\cbar_\alpha$
that enumerates $N_\alpha$.  For any such choice of $\cbar_{\alpha+1}$, since $\Mbar^{\alpha+1}\subseteq_{TV} \Mbar$, we have that $\cbar_{\alpha+1}$ is also a construction sequence over $\bigcup\Mbar$, as required.
\qed\end{proof}

\section{$\PPe$ and always isolated types}

Following an idea of Baisalov \cite{Baisalov}, we begin with a novel definition.

\begin{Definition}  \label{original} {\em  An {\em e-type} is a stationary, weight one type $p(x,d)$ with $d$ finite that is non-isolated.
}
\end{Definition}  

Following the template given in \cite{PNDOP}, we use this notion to define a family of regular types.  In \cite{Baisalov} he called elements of $\PPe$ w-types.  


\begin{Definition} {\em  $\PPe$ is the set of all regular types that are non-orthogonal to an $e$-type.  
}
\end{Definition}  

It is evident that the class $\PPe$ of regular types is closed under automorphisms of $\C$ and non-orthogonality.  The latter uses that non-orthogonality is an equivalence relation on the class of all stationary, weight one types.  With an eye on the results in \cite{PNDOP}, we show one more closure property of $\PPe$.

\begin{Definition}  {\em  Suppose $p,q$ are regular types.  We say  {\em  $q$ lies directly over $p$} if there are a-models $M\preceq N$  and elements $a,b$
such that $\tp(a/M)$ regular and non-orthogonal to $p$, $\tp(b/N)$ regular and non-orthogonal to $q$, with $N$ dominated by $a$ over $M$ and $q\perp M$.
We also say {\em $p$ supports $q$} if $q$ lies directly over $p$.  
}
\end{Definition}

\begin{Lemma}   \label{support}   If $q\in\PPe$ and $q$ lies directly over some regular type $p$, then $p\in\PPe$ as well.  
\end{Lemma}  

\begin{proof}  Let $M$ be an a-model on which $p$ is based and let $N=M[a]$ be a-prime over $M$ and a realization of $p$ with $q\not\perp N$.  As $q\in\PPe$
choose a non-isolated weight one, stationary $r\in S(d)$ with $d\subseteq N$ finite with $r\not\perp q$.  Let $b$ be any realization of $r|da$.  Since 
$da$ is dominated by $a$ over $M$ and $\tp(b/da)\perp M$, we have that $bda$ is dominated by $a$ over $M$, from which it follows that $wt(bd/M)=1$.
Choose any finite $e\subseteq M$ so that $\tp(bd/M)$ is based and stationary on $e$.  Then $\tp(bd/e)$ is a non-isolated weight one, stationary type non-orthogonal to $p$,
hence $p\in\PPe$.
\qed\end{proof}

Thus, in the terminology of \cite{PNDOP}, $\PPe={\bf P_e^{active}}$.  We now turn to the complementary notion.  

\begin{Definition}  {\em  A strong type $p$ is  {\em always isolated}  if, for all finite $d$ on which $p$ is based, $\tp(a/d)$ is isolated for any realization of $a$ of $p|d$.
}
\end{Definition}

\begin{Lemma}  \label{AIlemma}   For any strong type $p$, if $p\perp \PPe$, then $p$ is always isolated.
\end{Lemma}

\begin{proof}  By induction on $wt(p)$.  
If $wt(p)=1$, but there were some finite $d$ on which $p$ is  based  with $p|d$ non-isolated, then let $b$ be any realization of $p|d$.
Then $p|db$ is stationary, weight one, and non-isolated by the Open Mapping Theorem, contradicting $p\perp \PPe$, so the Lemma holds for $wt(p)=1$.
     
Assume the Lemma holds for all strong types of weight at most $n$.  Choose a strong type $p\perp \PPe$ of weight $n+1$ and choose a finite set $d$ on which $p$ is based.
Choose an a-model $M\supseteq d$ and let $a$ realize $p|M$.  We will show $tp(a/d)$ is isolated.  
As $M$ is an a-model, choose an $M$-independent set $\{b_i:i\le n\}$ such that $\tp(b_i/M)$ is regular and $\nfg a M {b_i}$ for each $i$.
Choose an a-model $N^*=M[b_i:i\le n]$ with $a\in N^*$ and within $N^*$, choose $N\preceq N^*$ to be a-prime over $M\cup\{b_i:i<n\}$.    
Choose any finite $e\subseteq N$ on which $\tp(a/N)$ is based.  As $e\subseteq N$, $e$ is dominated by $\{b_i:i<n\}$ over $M$, hence $\tp(e/M)\perp \PPe$
and $wt(e/M)\le n$.   Choose a finite $h\supseteq d$ on which $\tp(e/M)$ is based.  By our inductive hypothesis $\tp(e/h)$ is isolated.  
As well, $eh\subseteq N$ and $\tp(a/N)$ is based on this set, hence $\tp(a/eh)$ is isolated as well.  Putting these together, $\tp(a/h)$ is isolated.
However,  $\fg a d h$, so $\tp(a/d)$ is isolated by the Open Mapping Theorem.
\qed\end{proof}

\begin{Proposition}  \label{charAI}   Let $p$ be any strong type.  Then $p\perp \PPe$ if and only if every strong type $q\triangleleft p$ is always isolated (see Definition~\ref{domdef}).
\end{Proposition}

\begin{proof}  If $p\perp \PPe$, then $q\perp \PPe$ for every $q\triangleleft p$, so left to right follows from Lemma~\ref{AIlemma}.
For the converse, assume $p\not\perp \PPe$.  Choose a regular $q\in \PPe$ with $q\not\perp p$.   
As $q\in\PPe$, choose a  stationary, weight one $r\not\perp q$
that is not always isolated.  As $r\triangleleft q\triangleleft p$, we finish.
\qed\end{proof}

The following Lemma is likely well-known, but as we do not know of a specific reference, we include its proof for the convenience of the reader. 

\begin{Lemma}  \label{FEQ}   ($T$ stable)  Suppose $p=\tp(a/b)$ is any non-algebraic type and let $q(x)\in S(ba)$ be the stationary non-forking extension for which $\stp(c/b)=\stp(a/b)$ for
some/every realization $c$ of $q$.  If $q$ is isolated, then there are only finitely many strong types extending $p$.  In fact, there is some $E^*(x,y)\in FE(b)$ such that
for any $a_1,a_2$ realizing $p$, $\stp(a_1/b)=\stp(a_2/b)$ iff $E^*(a_1,a_2)$ holds.
\end{Lemma}

\begin{proof}  Let $\{\alpha_i(x,b,a):i\in I\}$ enumerate the formulas that fork over $b$ and let $\{E_j:j\in J\}$ enumerate $FE(b)$.  Note that $\bigwedge_{j\in J} E_j(x,a)\vdash p(x)$, 
so by the Finite Equivalence Relation theorem, $q$ is generated by 
$$\{E_j(x,a):i\in J\}\cup\{\neg\alpha_i(x,b,a):i\in I\}$$
If $\phi(x,b,a)$ isolates $q$, there are finite subsets $J_0\subseteq J$ and $I_0\subseteq I\}$ entailing $\phi(x,b,a)$.  Put $E^*(x,y):=\bigwedge_{j\in J_0}E_j(x,y)$.
Then $E^*(x,y)\in FE(b)$, and it suffices to show that $\forall x(E^*(x,a)\vdash E_j(x,a))$ for every $j\in J$.  To see this, fix any $j\in J$ and choose any $a_1$ such that $E^*(a_1,a)$ holds.  As $E^*(x,a)\in FE(b)$, it is not a forking formula over $b$, hence there is some $a_2\in \C$ such that $\stp(a_1/b)=\stp(a_2,b)$, $E^*(a_1,a_2)$, and $\fg {a_2} b a$.
From above, $\phi(a_2,b,a)$, hence $E_j(a_2,a)$ holds.  But $\stp(a_1/b)=\stp(a_2/b)$ implies $E_j(a_1,a_2)$, so $E_j(a_1,a)$, as required.
\qed\end{proof}

Lemma~\ref{FEQ} immediately gives the following $\omega$-stable-like behavior of types orthogonal to $\PPe$.

\begin{Lemma}   \label{stationary}
 Suppose $M$ is any model and $p=\tp(a/M)\perp \PPe$.     Then there is some finite $d\subseteq M$ on which $p$ is based and stationary.
\end{Lemma}

\begin{proof}  Choose a finite $b\subseteq M$ on which $p$ is based.  Since $p$ is always isolated, choose $\psi(x,b)$ isolating $p_0:=\tp(a/b)$.
Choose $a'\in p(\C)$ with $\fg a b {a'}$ and let $q=\tp(a/ba')$.  So $q$ is stationary and based on $ba'$.
As $q$ is parallel to $p$, it is always isolated, hence there is $\phi(x,b,a')$ isolating $q|ba'$.  Choose $E^*(x,y)\in FE(b)$ as in Lemma~\ref{FEQ}.  Now $a'\not\in M$, but $\psi(M,b)$ contains a complete set of representatives of  the $E^*(x,y)$-classes consistent with $\psi(x,b)$.  Choose $a^*\in \psi(M,b)$ with $E^*(a',a^*)$.
As both $a'$ and $a^*$ realize $p_0$, it follows from Lemma~\ref{FEQ} that $\stp(a'/b)=\stp(a^*/b)$.  Choose a strong automorphism $\sigma\in Aut(\C/b)$ with $\sigma(a')=a^*$
and put $q^*(x,b,a^*):=\sigma(q)$.   Then $q^*$ is stationary and, since
since  both $a'$ and $a^*$ are independent from $a$ over $b$, it follows that $\tp(aa'b)=\tp(aa^*b)$.  That is, $a$ realizes the stationary type $q^*$, so $p$ is based and stationary
over $ba^*$.
\qed\end{proof}

Recall that a stationary type $p\in S(A)$ is {\em strongly regular} if there is a formula $\phi(x,a)\in p$ such that for every global type $q\in S(\C)$ with $\phi(x,a)\in q$,
either $q\perp p$ or $q$ is the unique non-forking extension of $p$ to $S(\C)$.

\begin{Lemma} \label{SRna}   Suppose $p\in S(M)$ is regular with $p\perp \PPe$.  Then $p$ is both strongly regular and na.
\end{Lemma}

\begin{proof}  By Lemma~\ref{stationary}, choose any finite $d\subseteq M$ on which $p$ is based and stationary.  Since $p$ is always isolated,
there is a formula $\theta(x,d)\in p$ isolating the type $p|d$.  
We first argue that $p$ is strongly regular via $\theta(x,d)$.  To see this, choose any global type $q\in S(\C)$
with $\theta(x,d)\in q$.     As $\theta(x,d)$ isolates $p|d$, $q$ extends the regular type $p|d$.  Thus, if $q\not\perp p$,
then $q\not\perp p|d$, so $q$ is the non-forking extension of $p|d$, which is also the non-forking extension of $p$ to $S(\C)$.  


To see that $p$ is na, choose any $\phi(x,b)\in p$.  By replacing $\phi(x,b)$ by $\phi(x,b)\wedge\theta(x,d)$, we may assume $d\subseteq b$ and $\phi(x,b)\vdash \theta(x,d)$.
Since $p$ is always isolated, $p|b$ is isolated, say by $\delta(x,b)$.  Since $M$ is a model, choose $e\in M$ realizing $\delta(x,b)$.    Then $e\in\phi(M,b)$, but $\tp(e/b)=p|b$,
hence is non-algebraic.
\qed\end{proof}

Recall that a type $\tp(a/B)$ is {\em c-isolated} if  there is a formula $\phi(x,b)\in\tp(a/B)$ such that $R^\infty(q)=R^\infty(\phi(x,b))$ for every $q\in S(B)$ with $\phi(x,b)\in q$.

\begin{Proposition}  \label{DI}  Suppose $\stp(b/M)\perp \PPe$ and choose $c$ to both be $c$-isolated over $Mb$ 
and such that
$bc$ is dominated by $b$ over $M$.  Then $\tp(c/Mb)$ is isolated.
\end{Proposition}

\begin{proof}  
Choose $\theta(x,b,m)\in \tp(c/Mb)$ with $R^\infty(\theta(x,b,m))=R^\infty(c/Mb)$.  By increasing $m$ (but keeping it finite), Lemma~\ref{stationary} allows us to assume $\tp(bc/M)$ is based and stationary on $m$.  
%
%
%
As $\tp(bc/M)\perp \PPe$, we conclude that $\tp(bc/m)$ is isolated, say by $\phi(x,y,m)$.   
It follows that $\phi(x,b,m)$ isolates $\tp(c/bm)$, but we argue that $\phi(x,b,m)$ isolates $\tp(c/Mb)$.  

To see this,   choose any $c'$ realizing $\phi(x,b,m)$.   Note that $\phi(x,y,m)\vdash \theta(x,y,m)$, so it follows that $\fg {c'} {bm}  M$.
However, $\tp(c'b/m)=\tp(cb/m)$ is stationary and both $cb$ and $c'b$ are independent with $M$ over $m$, hence $\tp(c'bM/m)=\tp(cbM/m)$.  As $m\subseteq M$,
this implies $\tp(c'/Mb)=\tp(c/Mb)$, as required.
\qed\end{proof}

We do not know whether the following result requires that $M$ be countable.  

\begin{Corollary}  \label{prime}  For any countable model $M$ and any finite $b$, if $\tp(b/M)\perp \PPe$, then there is a constructible model $M(b)\supseteq Mb$.
\end{Corollary}

\begin{proof}  As $T$ and $M$ are countable, it suffices to show that the isolated types over $Mb$ are dense.  Choose any consistent formula $\psi(x,a,b)$ with $a$ from $M$.
Among all consistent  $L(Mb)$-formulas that imply $\psi(x,a,b)$, choose $\theta(x,a',b)\vdash \psi(x,a,b)$ of least $R^\infty$-rank and then choose
an element $c\in\theta(\C,a',b)$ that is $\ell$-isolated over $Mb$.  Thus, $cb$ is dominated by $b$ over $M$ and $\tp(c/Mb)$ is c-isolated, hence
$\tp(c/Mb)$ is isolated by
 Proposition~\ref{DI}. 
From this density result, a constructible model over $Mb$ exists.
\qed\end{proof}

%
%
%

\subsection{Dull pairs}

\begin{Definition} {\em  We call $M\preceq N$ a {\em dull pair} if $\tp(N/M)\perp \PPe$.
}
\end{Definition}

The following Proposition gives an easily checkable criterion for whether a given pair of models is dull.  
The notion of $M\subseteq_{na} N$ is defined in Definition~\ref{naDef}.

\begin{Proposition}  \label{chardull}    $M\preceq N$ is a dull pair if and only if every regular type $p\in S(M)$ realized in $N$ is orthogonal to $\PPe$.
\end{Proposition}

\begin{proof}  Left to right is immediate.  For the converse, assume $M\preceq N$ and that every regular type $p\in S(M)$ realized in $N$ is $\perp \PPe$.
By Lemma~\ref{SRna}, every such $p$ is na, so by Proposition~\ref{charna}, $M\subseteq_{na} N$.  
Assume by way of contradiction that $\tp(N/M)\not\perp \PPe$.  
Choose a (regular) $r\in \PPe$ such that $\tp(N/M)\not\perp r$.
Since $M\subseteq_{na} N$,  it follows from  e.g., Proposition~8.3.5 of \cite{Pillay} that  
there is $b\in N\setminus M$ such that $\tp(b/M)$ is regular and non-orthogonal to $r$, contradicting our assumption.
\qed\end{proof}

\begin{Lemma}  \label{dullfacts}   Suppose $M\preceq N$ is a dull pair.  Then:
\begin{enumerate}
\item $M\subseteq_{na} N$;
\item  There is some $a\in N\setminus M$ with $\tp(a/M)$ strongly  regular,  and for any such choice of $a$, 
\begin{enumerate}
\item  There is a model $M'\preceq N$ containing $a$, and dominated by $a$ over $M$; and
\item  For any such $M'$, $M'\preceq N$ is a dull pair.
\end{enumerate}
\item  If $(M_\alpha:\alpha<\beta)$ is any continuous, increasing chain of substructures of $N$ with $M_\alpha\preceq N$ dull for each $\alpha$, then
$\bigcup\{M_\alpha:\alpha<\beta\}\preceq N$ is a dull pair.
\end{enumerate}
\end{Lemma}

\begin{proof}  
(1)  Choose any regular type $p\in S(M)$ that is realized in $N$.  Since $p\perp \PPe$, $p$ is na by Lemma~\ref{SRna}, so $M\subseteq_{na} N$ by Proposition~\ref{charna}.

(2)  Choose any $a\in N\setminus M$ for which $\tp(a/M)$ is regular.  Then $\tp(a/M)$ is strongly regular by Lemma~\ref{SRna}.  Now fix such an $a\in N\setminus M$.  
For (a), the existence of such an $M'$ follows by (1) and Fact~\ref{3big}(3).
For (b), fix any such $M'$ and
choose any regular $p\in S(M')$ that is realized in $N$.  Any such $p$ is strongly regular by Lemma~\ref{SRna}.    We show
$p\perp\PPe$ by splitting into cases.  If $p\not\perp M$, then since $M\subseteq_{na} N$ by (1), Fact~\ref{3big}(2) (the 3-model Lemma) implies there is some regular 
$q\in S(M)$ non-orthogonal to $p$ that is realized in $N$.  Since $M\preceq N$ is dull, $q\perp \PPe$,  hence $p\perp \PPe$ as well.   On the other hand, if
$p\perp M$, then $p$ lies directly above $\tp(a/M)$, hence $p\perp \PPe$ by Lemma~\ref{support}.

(3)  Let $M^*:=\bigcup\{M_\alpha:\alpha<\beta\}$ and choose any regular type $p\in M^*$ that is realized in $N$.  By superstability, choose $\alpha^*<\beta$ such that
$p$ is based and stationary on $M_{\alpha^*}$.  Since $M_{\alpha^*}\preceq N$ is dull, $p|M_{\alpha^*}$ and hence $p$ are $\perp \PPe$.
\qed\end{proof}


\begin{Definition}  {\em  Suppose $M\preceq N$ are models.  A {\em strongly regular filtration of $N$ over $M$} is a continuous, elementary chain
$(M_\alpha:\alpha\le\beta)$ of models satisfying $M_0=M$, $M_\beta=N$, and for each $\alpha<\beta$, there is some $a_\alpha\in M_{\alpha+1}$ such that
$\tp(a_\alpha/M_\alpha)$ is strongly regular and $M_{\alpha+1}$ is dominated by $a_\alpha$ over $M_\alpha$.

We say that a strongly regular filtration is {\em prime} if, in addition, $M_{\alpha+1}$ is constructible over $Ma_\alpha$.
}
\end{Definition}

\begin{Proposition}  \label{filtration}  
Suppose $M\preceq N$ is dull.  Then a strongly regular filtration of $N$ over $M$ exists.  Additionally, if $N$ is countable, then a prime strongly regular filtration exists.
\end{Proposition}

\begin{proof}  Given $M\preceq N$, construct a maximal  continuous chain of submodels $(M_\alpha:\alpha<\gamma)$ of $N$ such that 
$M_0=M$, each $M_\alpha\preceq N$ is dull,
and $M_{\alpha+1}$ is dominated over $M_\alpha$ by some $a_\alpha\in N$ for which $\tp(a_\alpha/M_\alpha)$ is strongly regular.  
By Lemma~\ref{dullfacts}(3), $\gamma$ cannot be a limit ordinal, so say $\gamma=\beta+1$.  
We argue that $M_\beta=N$.  If this were not the case, then as $M_\beta\preceq N$ is dull, by Lemma~\ref{dullfacts}(2)
there is some $a_\beta\in N$ and $M_{\gamma}$ contradicting the maximality of the chain.

In the case where $N$ is countable, we argue similarly, but at stage $\alpha$, if we are given $M_{\alpha}$ and $a_\alpha\in N\setminus M_\alpha$ with $\tp(a_\alpha/M_\alpha)$
strongly regular, we use Corollary~\ref{prime} to choose a constructible model $M_\alpha(a_\alpha)$ over $M_\alpha a_\alpha$ to serve as $M_{\alpha+1}$.
\qed\end{proof}

Proposition~\ref{filtration} begets the following Corollaries.

\begin{Corollary}    \label{splitfiltration}  
\begin{enumerate} 
\item   Suppose $M\preceq N$ is dull and $K$ is any model satisfying $M\preceq K\preceq N$.  Then both $M\preceq K$ and $K\preceq N$ are dull.
\item    Suppose $M\preceq K$ is dull and $K\preceq N$ is dull.  Then $M\preceq N$ is dull.
\end{enumerate}
\end{Corollary}

\begin{proof}   (1)  Given $M\preceq K\preceq N$ with $M\preceq N$ dull, it is trivial that $M\preceq K$ is dull, hence has a strongly regular filtration $(M_\alpha:\alpha\le\beta)$ with
$M_\beta=K$.  By Lemma~\ref{dullfacts}(2,3), depending on whether or not $\beta$ is a limit ordinal, $K\preceq N$ as well.

(2)  Now suppose $M\preceq K$ and $K\preceq N$ are dull.  By Proposition~\ref{filtration} on each part, there is a  filtration of $K$ over $M$ and a filtration of $N$ over $K$.
The concatenation of these filtrations gives a filtration $(M_\alpha:\alpha\le\beta)$ of $N$ over $M$.  Note that for every $\alpha<\beta$, $\tp(a_\alpha/M_\alpha)\perp \PPe$,
either because $M\preceq K$ is dull (when $M_\alpha\preceq K$) or because $K\preceq N$ is dull (when $K\preceq M_\alpha)$.
To see that $M\preceq N$ is dull, choose any $e\in N\setminus M$ for which
$q=\tp(e/M)$ is regular.  As $e\in M_\beta=N$, there is a least $\alpha$ such that $\fg e M {M_\alpha}$, but $\nfg e {M_{\alpha}} {M_{\alpha+1}}$.  
Then $q$ is non-orthogonal to $\tp(a_\alpha/M_\alpha)$, the latter being regular and $\perp \PPe$ by our sentences above.  
Thus $M\preceq N$  is dull by Proposition~\ref{chardull}.
\qed\end{proof}

We now embark on a series of Lemmas that will lead us to Proposition~\ref{iso} and Theorem~\ref{absolute}.  The structure of this argument is similar to what appears in
Section~1 of \cite{Borel}, but there we were assuming the theory was $\omega$-stable.  
Indeed, these  Lemmas would have  much easier proofs under the assumption of  $(\aleph_0,2)$-existence, but here we are not even assuming this.

\begin{Lemma} \label{atomicexists}   Suppose $N$ is countable and $M\preceq N$ is dull.  Let $J\subseteq N\setminus M$ be any $M$-independent set of finite sets.
Then a constructible model over $MJ$ exists.
\end{Lemma}  

\begin{proof}  Enumerate $J=\{\abar_i:i\in\omega\}$ and, for each $n\in\omega$, let $J_n=\{\abar_i:i<n\}$.  
For each $n\in\omega$, choose $M_n$ to be constructible over $MJ_n$.  Without loss, we may assume
$M_n\preceq M_{n+1}$ for each $n$.  We claim that  $M^*=\bigcup\{M_n:n\in\omega\}$ is atomic over $MJ$.  To see this, choose any $\ebar\subseteq M^*$.
Choose $n\in\omega$ so that $\ebar\subseteq M_n$.
As $\fg {J_n} M {(J\setminus J_n)}$ we have 
$MJ_n\subseteq_{TV} MJ$, hence $\tp(\ebar/MJ)$ is isolated by Lemma~\ref{TVup}.  Thus, $M^*$ is atomic over $MJ$.  But as $M^*$ is countable,
it is also constructible over $MJ$ by Fact~\ref{arbitrary}(2).   
\qed\end{proof}

\begin{Lemma}    \label{J1constructible}  Suppose $N$ is countable and $M\preceq N$ is dull.  If $J\subseteq N\setminus M$ is any maximal $M$-independent set, then
$N$ is constructible over $MJ$.
\end{Lemma}

\begin{proof}  By Lemma~\ref{atomicexists} there is some constructible model $M^*$ over $MJ$, which we may assume is contained in $N$.
Thus, $N$ contains an atomic model over $MJ$.  By Zorn's Lemma, let $N_0\preceq N$ be any maximal atomic model over $MJ$ contained in $N$.
By Fact~\ref{arbitrary}(2)  it suffices to show that $N_0=N$.  Assume by way of contradiction that $N_0\prec N$ is proper.
 Choose any $e\in N\setminus N_0$ such that
$q=\tp(e/N_0)$ is regular.  
By Corollary~\ref{splitfiltration}(1), $N_0\preceq N$ is dull, hence $q\perp \PPe$.  On one hand, if $q\not\perp M$, then as $M\subseteq_{na} N$ by 
Lemma~\ref{dullfacts}(1), the 3-model lemma gives $d\in N$ with $\fg d M {N_0}$, but this contradicts the maximality of $J$.  
So assume $q\perp M$.  We will obtain a contradiction by showing that $N_0e$ is atomic over $MJ$.   For this, it suffices to show that $\tp(de/MJ)$ is isolated for any finite $d\subseteq N_0$.
So choose any finite $d\subseteq N_0$.  Choose a finite $d^*$, $d\subseteq d^*\subseteq N_0$ on which $q$ is based and stationary.  By superstablity, choose
 any finite $J^*\subseteq J$ for which $\fg {d^*} {MJ^*}  J$.  Note that $d^*J^*$ is finite and $\fg {d^*J^*} M {(J\setminus J^*)}$.
 Choose any finite $\abar^*\subseteq M$ such that $\fg {d^*J^*} {a^*} M$.  It follows by transitivity of non-forking that
 $$\fg {d^*J^*} {a^*} {M(J\setminus J^*)}$$
 As $d^*J^*a^*\subseteq N_0$,  $\tp(e/d^*J^*a^*)=q|d^*J^*a^*$, and since $q\perp M$ (hence $q\perp a^*$) it follows from Fact~\ref{basicorth}(2)  that 
 $$\tp(e/d^*J^*a^*)\vdash \tp(e/d^* MJ)$$
 However, since $\tp(e/N_0)$ is always isolated, $\tp(e/d^*J^*a^*)$ is isolated, so $\tp(e/d^*MJ)$ is isolated as well.  Since $d^*\subseteq N_0$, $\tp(d^*/MJ)$ is also isolated. 
 Thus, $\tp(ed^*/MJ)$ and hence $\tp(ed/MJ)$ is isolated as well.  This contradicts the maximality of $N_0$.
\qed\end{proof}

\begin{Lemma}  \label{together}   Suppose $N$ is countable and $p\in S(N)$ is regular and $\perp \PPe$.  Then for any finite set $A\subseteq N$, $N\cong_A N(c)$,
where $c$ is any realization of $p$ and $N(c)$ is any constructible model over $Nc$.
\end{Lemma}

\begin{proof}  Note that by Corollary~\ref{prime}, a constructible model $N(c)$ exists.  Choose any finite $A\subseteq N$, and by increasing $A$ if necessary, we may
assume $p$ is based and stationary on $A$.  As $p$ is always isolated, there is an infinite Morley sequence $J\subseteq N$ in $p|A$.
Partition $J=J_0\cup J_1$ into two infinite pieces and choose $B\subseteq N$ maximal such that
$\fg B {AJ_0} {J_1}$.

\medskip
\noindent{\bf Claim 1.}  $B$ is the universe of an elementary submodel $M\preceq N$.

\begin{proof}  If not, choose  $\alpha$ least for which $R^\infty(\phi(x,b))=\alpha$ for some  consistent $L(B)$-formula $\phi(x,b)$  that is not realized in $B$.  
Since $N$ is a model, choose $e\in \phi(N,b)$.
We argue that $\fg  e {B} {J_1}$, which contradicts the maximality of $B$.   
Assume by way of contradiction that this non-forking failed.  Then there would be some finite $b^*\subseteq B$ such that $\nfg {e}  {AJ_0b^*} {J_1}$, and we may assume
$b^*\supseteq b$.  Choose a forking formula $\delta(x,a,b^*,\hbar_0, \hbar_1)\in\tp(e/AJ_0b^*J_1)$, where $a\in A$, $\hbar_0\subseteq J_0$ and  $\hbar_1\subseteq J_1$. 
By replacing $\delta$ by $\delta\wedge\phi$, we may assume $\delta(x,a,b^*,\hbar_0,\hbar_1)\vdash \phi(x,b)$, yet $R^\infty(\delta(x,a,b^*,\hbar_0,\hbar_1))<\alpha$.
We will obtain a contradiction by `dropping $\hbar_1$ down into $J_0$,' thereby contradicting the minimality of $\alpha$. 

By superstability again, by enlarging $\hbar_0\subseteq J_0$  we may additionally assume that $\fg {b^*} {A\hbar_0} {J_0}$, hence also
$$\fg {b^*}  {A\hbar_0} {J_0J_1}$$
by transitivity of non-forking.  
As $J=J_0\cup J_1$ is indiscernible over $A$ and $J_0$ is infinite, we can find some $\hbar_0'\subseteq J_0$ such that
$$\stp(\hbar_0'/A\hbar_0)=\stp(\hbar_1/A\hbar_0)$$
Thus, as both $\hbar_1$ and $\hbar_0'$ are independent from $b^*$ over $A\hbar_0$, $\tp(b^*a\hbar_0\hbar_0')=\tp(b^*a\hbar_0\hbar_1)$.
It follows that $R^\infty(\delta(x,b^*, a,\hbar_0,\hbar_0'))<\alpha$ and is a consistent $L(B)$-formula that is not realized in $B$.   This contradicts our choice of $\phi(x,b)$.
\qed\end{proof}

\medskip
\noindent{\bf Claim 2.}  $J_1$ is a maximal $M$-independent subset of $N\setminus M$.  

\begin{proof}  Choose any $e\in N$ with $\fg e M {J_1}$.  Since $\fg M {AJ_0} {J_1}$, $Me$ is independent from $J_1$ over $AJ_0$, hence $e\subseteq M$ by the maximality of $M$.
\qed\end{proof}

It follows from Claim 2 that $M\preceq N$ is dull.  To see this, choose  $e\subseteq N$ with $\tp(e/M)$ regular.  Since $J_1$ is  Morley sequence in $p|M$, the fact
that $\nfg e M {J_1}$ implies that $\tp(e/M)\not\perp p$, hence $\tp(e/M)\perp \PPe$.  Thus $M\preceq N$ is dull by Proposition~\ref{chardull}.  
From this, it follows from 
Lemma~\ref{J1constructible} that $N$ is constructible over $MJ_1$.  

\medskip
\noindent{\bf Claim 3.}  $N(c)$ is constructible over $MJ_1c$.

\begin{proof}  We know $N$ is constructible over $MJ_1$ and since $\fg c {M} {J_1}$, we have $MJ_1\subseteq_{TV} MJ_1c$, thus by Lemma~\ref{TVup},
the universe of $N$ is a construction sequence over $MJ_1c$.  As $N(c)$ is constructible over $Nc$, it follows that the concatenation of these two sequences is
a construction sequence of $N(c)$ over $MJ_1c$.
\qed\end{proof}

Finally, since both $J_1$ and $J_1c$ are infinite Morley sequences in $p|M$, any bijection $f_0:J_1\rightarrow J_1c$ extends to an elementary map
$f:MJ_1\rightarrow MJ_1c$ with $f\mr{M}=id$.   By the uniqueness of constructible models, it follows that $f$ can be extended to an isomorphism $f^*:N\rightarrow N(c)$
fixing $M$ (and hence $A$) pointwise.
\qed\end{proof}

\begin{Definition}  {\em  Given two structures $M,N$ and finite tuples of the same length $\abar\in M^k$, $\bbar\in N^k$,
we say $\tp^\infty_M(\abar)=\tp^\infty_N(\bbar)$ if the structures $(M,\abar)$ and $(N,\bbar)$ are back and forth equivalent.

We say  $M\preceq_{\infty,\omega} N$ if,  $M\subseteq N$ and  $\tp_M^\infty(\abar)=\tp_N^\infty(\abar)$ for all finite $\abar\in M^{<\omega}$
}
\end{Definition}  

It is evident that the relation $\precbf$ is transitive.   Moreover, 
when $M\subseteq N$ are countable, then  for $\abar\in M^k$, $\bbar\in N^k$, $\tp^\infty_M(\abar)=\tp^\infty_N(\bbar)$ if and only if there is an isomorphism
$f:M\rightarrow N$ with $f(\abar)=\bbar$.

We record the following easy Lemma.

\begin{Lemma}  \label{bfchain}  Suppose $\{M_n:n\in\omega\}$ are countable with $M_n\precbf M_{n+1}$ for every $n\in\omega$ and let $M^*=\bigcup\{M_n:m\in\omega\}$.
Then $M_n\precbf M^*$ for every $n\in\omega$.
\end{Lemma}

\begin{proof}  Note that since $M_\ell\precbf M_{\ell'}$ whenever $\ell\le\ell'<\omega$, $\tp^\infty_{M_\ell}(\bbar)=\tp^\infty_{M_{\ell'}}(\bbar)$ for every $\bbar\subseteq M_\ell$.
It suffices to show that $M_0\precbf M^*$.  For this, we claim that
$$\F=\{(\abar,\bbar):\abar\subseteq M_0, \bbar\subseteq M^*,\ \hbox{and $\tp^\infty_{M_0}(\abar)=\tp^\infty_{M_\ell}(\bbar)$ whenever $\bbar\subseteq M_\ell$}\}$$
is a back-and-forth system.  

Since $(\abar,\abar)\in\F$ for any $\abar\subseteq M_0$, $\F$ is nonempty.
Next, choose $(\abar,\bbar)\in\F$ and choose $c\in M_0$.  Choose $\ell$ such that $\bbar\subseteq M_\ell$. As $\tp^\infty_{M_0}(\abar)=\tp^\infty_{M_\ell}(\bbar)$,
choose an isomorphism $f:M_0\rightarrow M_\ell$ with $f(\abar)=\bbar$.  Put $d:=f(c)$.  Then $(\abar c,\bbar d)\in \F$.  Finally, choose 
$(\abar,\bbar)\in\F$ and choose $d\in M^*$.  Choose $\ell$ such that $\bbar d\subseteq M_\ell$.  As $\tp^\infty_{M_0}(\abar)=\tp^\infty_{M_\ell}(\bbar)$,
choose an isomorphism $g:M_\ell\rightarrow M_0$ with $g(\bbar)=\abar$.  Then taking $c:=g(d)$ yields $(\abar c,\bbar d)\in \F$.

GIven that $\F$ is a back and forth system, showing  $M_0\precbf M^*$ is easy.  Choose $\abar\in M_0^k$.   Since $(\abar,\abar)\in\F$ and both $M_0$ and $M^*$ are countable,
there is an isomorphism $h:M_0\rightarrow M^*$ with $h\mr{\abar}=id$, as required.
\qed\end{proof}

\begin{Proposition}  \label{iso}  Suppose  $N$ is countable and $M\preceq N$ is dull.  Then for any finite  set $A\subseteq M$,  $M\cong_A N$.  In particular, $M$ and $N$ are isomorphic.  
\end{Proposition}  

\begin{proof}  By Proposition~\ref{filtration}, choose a prime, strongly regular filtration $(M_\alpha:\alpha\le\beta)$ of $N$ over $M$.
As $M_{\alpha+1}$ is constructible over $M_\alpha a_\alpha$ with $\tp(a_\alpha/M_\alpha)\perp \PPe$, it follows from Lemma~\ref{together}
that $M_\alpha\precbf M_{\alpha+1}$ for each $\alpha<\beta$.  Additionally, for any countable limit ordinal $\gamma<\beta$, $M_\alpha\precbf M_\gamma$
for all $\alpha<\gamma$.  That is, $(M_\alpha:\alpha\le\beta)$ is a $\precbf$-chain of countable models with with $M_0=M$ and $M_\beta=N$.
Thus, $M\precbf N$, which suffices.
\qed\end{proof}  

\begin{Theorem}  \label{absolute}  Suppose (only) that $T$ is countable and superstable.
Suppose $M\preceq N$ is dull with $M,N$ of arbitrary size.  Then for any finite set $A\subseteq M$, 
$(M,a)_{a\in A}\equiv_{\infty,\omega}(N,a)_{a\in A}$.  In particular, $M$ and $N$ are back and forth equivalent.
\end{Theorem}

\begin{proof}  Fix any finite $A\subseteq M$ and choose an enumeration $\abar$ of $A$.  We show that for all ordinals $\alpha$, 
$(M,\abar)\equiv_\alpha (N,\abar)$, i.e., are $\alpha$-equivalent.  To see this,
pass to any forcing extension $\V[G]$ of $\V$ in which $M$ and $N$ are countable.  
It is easily checked that $M\preceq N$ remains dull in $\V[G]$.
Thus, by Proposition~\ref{iso}, there is an isomorphism $f:M\rightarrow N$ fixing $A$ pointwise.  
The existence of $f$ implies that $(M,\abar)\equiv_\alpha (N,\abar)$ in $\V[G]$ for all ordinals $\alpha$.
By absoluteness it follows that this holds in $\V$ as well. 
\qed\end{proof}

We close the subsection by showing that  that dull pairs can be amalgamated, with no new non-orthogonality classes of regular types being realized.  

\begin{Lemma}   \label{dullamalg}  Suppose $M\preceq N_1$ and $M\preceq N_2$ are both dull pairs with $\fg {N_1} M {N_2}$.
Then there is $N^*$ for which  $N_1\preceq N^*$, $N_2\preceq N^*$ and $M\preceq N^*$ are all dull pairs.   Moreover, for any $e\in N^*\setminus M$ with $p:=\tp(e/M)$ regular,
either there  some is $h\in N_1$ or there is some $h\in N_2$ with $\tp(h/M)$ regular and $\not\perp p$.
\end{Lemma}  

\begin{proof}  As $T$ is countable and superstable, choose an $\ell$-contructible model  $N^*$  over $N_1N_2$.  We first show that no new non-orthogonality classes of regular types over $M$
are realized in $N^*$.  To see this, choose $e\in N^*\setminus M$ with $p:=\tp(e/M)$ regular.  Since $\fg {N_1} M {N_2}$ and since $\wt(e/M)=1$, $e$ forks with either $N_1$ or $N_2$ over $M$.   By symmetry, assume the former.  Then $p\not\perp \tp(N_1/M)$, and since $M\preceq N_1$ dull implies $M\subseteq_{na} N_1$ by Lemma~\ref{dullfacts}(1),
it follows from Fact~\ref{3big}(1) that there is $h\in N_1$ such that $\tp(h/M)$ is regular and $\not\perp p$.

In particular, since every regular type $\tp(h/M)\perp \PPe$, for every $h\in N_1$ and every $h\in N_2$, the regular type $p$ is also $\perp\PPe$. 
Thus, by Proposition~\ref{chardull}, $M\preceq N^*$ is a dull pair.  That $N_1\preceq N^*$ and $N_2\preceq N^*$ are dull as well follows from Corollary~\ref{splitfiltration}(1).
\qed\end{proof}

\subsection{$\PPe$-NDOP}

We begin with some general comments about DOP witnesses.  
In \cite{PNDOP}, this was extensively studied for regular types.   The following definition appears as Definition~3.1 of \cite{PNDOP}.  

\begin{Definition}   \label{origDOP}
 {\em A regular type $r$ {\em has a DOP witness}  if there is an independent triple $(M_0,M_1,M_2)$ of a-models and an a-prime model $M_3$ over $M_1M_2$
such that the canonical base $Cb(r)\subseteq M_3$, with $r\perp M_1$ and $r\perp M_2$.
}
\end{Definition}

Clearly, if a regular type $r$ has a DOP witness, then every stationary, weight one type $p(x,d)$ non-orthogonal to $r$ has a DOP witness as well.  However, the dependence on a-models make the definition a bit awkward to use.  The following two Lemmas use  $V$-domination to get more malleable conditions.  

\begin{Lemma}  \label{relaxamodel}  Let $\Abar=(A_0,A_1,A_2)$ be any independent triple, let $A^*$ be $V$-dominated by $\Abar$ and let
$p(x)$ be any stationary, weight one type whose non-orthogonality class does not have a DOP witness.  
If $p\not\perp A^*$, then $p\not\perp A_1$ or $p\not\perp A_2$.
\end{Lemma}

\begin{proof}  By way of contradiction, suppose $p\not\perp A^*$, but $p\perp A_1$ and $p\perp A_2$.  We will obtain a contradiction by constructing a DOP witness for
some regular type $r$ non-orthogonal to $p$.  
Suppose $p$ is based and stationary on the finite set $d$.  
Choose any $e$ such that $p\not\perp\stp(e/A^*)$ (where $\stp(e/A^*)$ need not be regular). 
Choose a finite $b\subseteq A^*$  such that $\fg e b {A^*}$.   By Fact~\ref{Vdomfact},(3) choose a finite $\Bbar\sq\Abar$ such that $b$ is $V$-dominated by $\Bbar$.  
Now let $\Mbar$ be any independent triple of $a$-models such that $\Bbar\sq\Mbar$ and $\fg d {B_1B_2} {M_1M_2}$.  
Note that for $\ell=1,2$, $p\perp A_\ell$ implies $p\perp B_\ell$.  However,  from $\Bbar\sq\Mbar$ and $\fg d {B_1B_2} {M_1M_2}$,
it follows that $\fg d {B_\ell} {M_\ell}$, hence $p\perp M_\ell$ as well.

Since $b$ was $V$-dominated by $\Bbar$ and since $\Bbar\sq\Mbar$, we have that $\tp(b/M_1M_2)$ is a-isolated.  Thus, we can find an a-prime model $M_3$ over $M_1M_2$
with $b\subseteq M_3$.  As $p\not\perp b$, $p\not\perp M_3$, hence there is a regular type $r\in S(M_3)$ non-orthogonal to $p$.  Then $(M_0,M_1,M_2,M_3)$ and $r$ form a DOP witness, giving our contradiction.
\qed\end{proof}

\begin{Definition}  \label{newDOP} {\em  A stationary type $p(x,d)$ over a finite set has a {\em finitary DOP witness} if there is an independent triple $\Abar=(a,b,c)$ of finite sets satisfying:
\begin{enumerate} 
\item   $d$ is $V$-dominated by $\Abar$;
\item  $p(x,d)\perp b$ and $p(x,d)\perp c$.
\end{enumerate}
}
\end{Definition}

\begin{Lemma}  \label{Dopwitness}  Suppose $p(x,d)$ is a stationary, weight one type whose non-orthogonality class has a DOP witness. 
Then $p(x,d)$ has a finitiary DOP witness.
\end{Lemma}

\begin{proof}  Choose a regular type $r\not\perp p(x,d)$ and choose a quadruple $(M_0,M_1,M_2,M_3)$ of a-models exemplifying 
 that $r$ has a DOP witness, and
fix a finite $e\subseteq M_3$ on which $r$ is based and stationary.
First, as a special case, assume $d\subseteq M_3$.  Then, as $\tp(d/M_1M_2)$ is $a$-isolated, $d$ is $V$-dominated by $\Mbar$ via Fact~\ref{3.3}.  
Choose a finite $h\subseteq M_1M_2$ on which $\tp(d/M_1M_2)$ is based and by Fact~\ref{category}(3), choose a finite $\Abar\sq \Mbar$ with $h\subseteq A_1A_2$.
As $\fg d {A_1A_2} {M_1M_2}$, it follows from Fact~\ref{Vdomfact}(3) that $d$ is V-dominated by $\Abar$.
Since $p(x,d)$ is weight one, non-orthogonal to $r$,  and as $r\perp M_1$ with $b\subseteq M_1$, it follows that $p(x,d)\perp b$.  Dually, $p(x,d)\perp c$, so $\Abar$ is a finitary DOP witness
for $p(x,d)$.


Now for the general case, since $M_3$ is an a-model, choose $d'\subseteq M_3$ with $\stp(d/e)=\stp(d'/e)$ and let $p'(x,d')$ be the conjugate type to $p(x,d)$ over $e$.
Since $r$ is based and stationary on $e$, we have that $p'(x,d')$ is stationary, weight one, and non-orthogonal to $r$.  As $d'\subseteq M_3$, apply the special case above
to get $\Abar=(a,b,c)$ for $p'(x,d')$.   Then, take any automorphism $\sigma$ of $\C$ fixing $\acl(e)$ pointwise, with $\sigma(d')=d$.  Then $\sigma(\Abar)=(\sigma(a),\sigma(b),\sigma(c))$ is a finitary DOP witness for $p(x,d)$.
\qed\end{proof}

\begin{Definition}  {\em  A  countable, superstable theory $T$ has {\em $\PPe$-DOP} if some (regular) $r\in \PP_e$ has a DOP witness.  We say $T$ has $\PPe$-NDOP
if it does not have $\PPe$-DOP.
}
\end{Definition}  

\begin{Proposition}  \label{VDINDOP}    If $T$ has V-DI, then $T$ has $\PPe$-NDOP.
\end{Proposition}

\begin{proof}  By way of contradiction, assume  some $r\in\PPe$ has a DOP witness, and that V-DI holds.   Choose a stationary, weight one $p(x,d)\not\perp r$ with $d$ finite and
$p(x,d)$ non-isolated.  By Lemma~\ref{Dopwitness}, find a finitary DOP witness 
$\Abar=(a,b,c)$ for $p(x,d)$.    Let $p'(x,dbc)\in S(dbc)$ be the non-forking extension of $p(x,d)$  and let $e$ realize $p'(x,dbc)$.
Since $d$ is $V$-dominated by $\Abar$ with $p(x,d)$ orthogonal to $b$ and $c$,  Lemma~\ref{adde} implies that $de$ is also  $V$-dominated by $\Abar$.   
Thus, by V-DI, $\tp(de/bc)$ is isolated, hence $\tp(e/dbc)=p'(x,dbc)$ is isolated as well.   As $p'(x,dbc)$ is a nonforking extension of $p(x,d)$,
 this contradicts the Open Mapping Theorem.
\qed\end{proof}

We close this section by summarizing our results so far, the equivalence of the first three conditions of Theorem~\ref{main}.

\begin{Theorem}  \label{sofar}  The following are equivalent for a countable, superstable theory $T$.
\begin{enumerate}
\item  V-DI;
\item  $\PPe$-NDOP and PMOP;
\item  $\PPe$-NDOP and countable PMOP;
\end{enumerate}
\end{Theorem}

\begin{proof}  $(1)\Rightarrow(2)$ is Theorem~\ref{DIPMOP} and Proposition~\ref{VDINDOP}, and $(2)\Rightarrow(3)$ is trivial.
So assume $T$ has $\PPe$-NDOP and countable PMOP.   Choose any independent triple $\Nbar=(N_0,N_1,N_2)$ of a-models and
assume $c$ is $V$-dominated by $\Nbar$.  We will show that $\tp(c/N_1N_2)$ is isolated.  
For this, first note that  $\tp(c/N_1N_2)$ is a-isolated by Fact~\ref{Vdomfact}.  Choose an a-prime model $N^*$ over $N_1N_2$ containing $c$.  

Construct, as a nested union of an $\omega$-chain of finite sets, a countable $M^*\preceq N^*$ such that, letting $M_i:=M^*\cap N_i$ for $i\in\{0,1,2\}$ we have
\begin{enumerate}
\item  $\fg {M^*}  {M_1M_2} {N_1N_2}$;
\item  $M_i\subseteq_{na} N_i$ for each $i$;
\item  $\Mbar=(M_0,M_1,M_2)$ is a (countable) independent triple of models and $\Mbar\sq\Nbar$;
\item  $c\subseteq M^*$.
\end{enumerate}

[To get the non-forking conditions, note that by superstablity, for every finite $d$ from $N^*$, there is a finite $X_d\subseteq N_1N_2$ for which 
$\fg d {A_1A_2} {N_1N_2}$ whenever $X_d\subseteq A_1A_2\subseteq N_1N_2$.

Given such an $M^*$, letting  $\Mbar=(M_0,M_1,M_2)$, we have $M^*$ is V-dominated by $\Mbar$ by Fact~\ref{Vdomfact}(3).
By countable PMOP, choose $M'\preceq M^*$, constructible
over $M_1M_2$.   

 \medskip
 \noindent{{\bf Claim 1.}}  For any $a\subseteq M^*\setminus M'$, if $\tp(a/M')$ is regular, then $\tp(a/M')\perp M_1$ and $\perp M_2$.
 
 \begin{proof}   Suppose $p=\tp(a/M')$ is regular.  By symmetry, it suffices to show $p\not\perp M_1$.  For this, 
 note that since $M^*$ is $V$-dominated by $\Mbar$, $M^*$ is dominated by $M_2$ over $M_1$.  [Why?  Choose any $Y$ with $\fg {M_2} {M_1} Y$.  It follows that $\Mbar\sq (M_0,M_1Y,M_2)$, hence $\fg {M^*} {M_1M_2} Y$ by $V$-domination.  Thus, $\fg {M^*} {M_1}
{M_2Y}$ by transitivity.]

Now, by way of contradiction, suppose $p\not\perp M_1$.  
Since $M_1\subseteq_{na} M^*$, the  3-model Lemma (Fact~\ref{3big}(2)) applied to $M_1\preceq M'\preceq M^*$ gives some $h\in M^*$ with $\fg h {M_1} {M^*}$, hence
$\fg h {M_1} {M_2}$, contradicting the domination described above. 
\qed\end{proof}

 \medskip
 \noindent{{\bf Claim 2.}}   $M'\preceq M^*$ is a dull pair, and $\tp(c/M')$ is orthogonal to $\PPe$, $M_1$, and $M_2$.  
 
 \begin{proof}  Choose any $a\in M^*\setminus M'$ with $p=\tp(a/M')$ regular.  Since $M'$ is $V$-dominated by $\Mbar$ and $T$ has $\PPe$-NDOP, 
 it follows from Lemma~\ref{relaxamodel}
 and Claim 1 that $p\perp \PPe$.  Thus,   $M'\preceq M^*$ is a dull pair by Proposition~\ref{chardull}.  It follows by Lemma~\ref{dullfacts}(1) that $M'\subseteq_{na} M^*$.
 
 Concerning the orthogonality, first suppose there were some $q\in\PPe$ with $q\not\perp \tp(c/M')$.  Since $M'\subseteq_{na} M^*$, it follows from Fact~\ref{3big}(1)
 that there is some $a\in M^*\setminus M'$ with $\tp(a/M')\in \PPe$, contradicting $M'\preceq M^*$ a dull pair.  Similarly, suppose $\tp(c/M')\not\perp M_1$.
 Then there would be some regular type $q\not\perp \tp(c/M')$ with $q\not\perp M_1$.  Since $M_1\subseteq_{na} M^*$, there again would be $a\in M^*\setminus M'$
 with $\tp(a/M')$ regular and $\not\perp q$, hence $\not\perp M_1$, contradicting Claim 1.  Showing $\tp(c/M')\perp M_2$ is symmetric.
 \qed\end{proof}
 
 As $\tp(c/M')\perp \PPe$, choose a finite $b\subseteq M'$ over which $\tp(c/M')$ is based and stationary.  
 By Claim 2 and Lemma~\ref{basicorth2}(1), $M_1M_2$ is essentially finite with respect to $\tp(c/b)$ (see Definition~\ref{essentiallyfinite})
 hence there is some finite $e\subseteq M_1M_2$ for which
 $$\tp(c/be)\vdash \tp(c/bM_1M_2)$$
 By Proposition~\ref{charAI}, $\tp(c/be)$ is isolated, hence $\tp(c/bM_1M_2)$ is isolated as well.  However,
 since $b\subseteq M'$, $\tp(b/M_1M_2)$ is also isolated, hence so are $\tp(bc/M_1M_2)$ and $\tp(c/M_1M_2)$.
 Finally, $\Mbar\sq\Nbar$, so $M_1M_2\subseteq_{TV} N_1N_2$, hence $\tp(c/N_1N_2)$ is isolated by Lemma~\ref{TVup}.
 \qed\end{proof}

\section{Tree decompositions}

The material in this section closely resembles Section~5 of \cite{Borel}, but here we are not assuming that $T$ is $\omega$-stable.  
However, from our work above, we see that under the assumption of V-DI (or any of its equivalents given in Theorem~\ref{sofar})
enough of the consequences of $\omega$-stability hold to make the arguments 
in \cite{Borel} go through.   Some of the earlier results of this section only require weaker hypotheses, such as PMOP.

\begin{Definition}  {\em A {\em tree} $(I,\trianglelefteq)$  is a non-empty, downward closed of $Ord^{<\omega}$, ordered by initial segment.
An {\em independent tree of models} is a sequence $\Mbar=(M_\eta:\eta\in I)$ of models, indexed by a tree $(I,\trianglelefteq)$ for which
$\fg {M_\eta} {M_{\eta^-}} {\bigcup\{M_\mu:\mu\not\triangleright\eta\}}$ for all $\eta\neq\<\>$.
To ease notation, for any subtree $J\subseteq I$, we write $M_J$ for $\bigcup\{M_\eta:\eta\in J\}$.}
\end{Definition}

Any independent tree of models is a stable system, hence analogues of Facts~\ref{category} and \ref{Vdomfact} apply to this case.  
Our first Lemma is an easy inductive construction.

\begin{Lemma}  \label{usePMOP}
Suppose $T$ is countable, superstable, with PMOP.  Then there is a constructible model over every independent tree $\Mbar=(M_\eta:\eta\in I)$.
\end{Lemma}

\begin{proof}  Choose any well ordering $I=(\eta_\alpha:\alpha<\delta)$
 such that $\eta_\alpha\triangleleft \eta_\beta$ implies $\alpha<\beta$.  
 
 This implies that for all $\alpha<\delta$, $I_{<\alpha}=\{\eta_\beta:\beta<\alpha\}$ is a subtree of $(I,\trianglelefteq)$, thus
 $$\Mbar_{<\alpha}:=\bigcup\{M_\gamma:\gamma<\alpha\}\subseteq_{TV}  \bigcup \Mbar_{<\beta}:=\{M_\gamma:\gamma<\beta\}\subseteq_{TV} M_I$$
 for all $\alpha<\beta<\delta$.
 It follows that any construction sequence $\cbar$  over $\Mbar_{<\alpha}$ is also a construction sequence over 
     $\Mbar_{<\beta}$  and over $M_I$. 
 
 We recursively find a sequence $(\cbar_\alpha:\alpha\le\delta)$ of sequences satisfying
 \begin{enumerate}
 \item  $\cbar_\alpha$ enumerates a constructible model $N_\alpha$ over $\bigcup\Mbar_{<\alpha}$; and
 \item  $\cbar_\alpha$ is an initial segment of $\cbar_\beta$ for all $\alpha\le\beta\le\delta$.
 \end{enumerate}
 If we succeed, then $\cbar_\delta$ will be a construction sequence over $M_I$.
 Put $\cbar_0:=\<\>$ and, for all non-zero limit ordinals $\gamma$, take $\cbar_\gamma$ to be the concatenation of $(\cbar_\alpha:\alpha<\gamma)$.
 Assuming $\cbar_\alpha$ has been chosen with $\alpha<\delta$, note that $M_\alpha$ is a `leaf' of the subtree $\Mbar_{\le\alpha}$.  
 As well, $$\fg {M_\alpha}  {M_{\alpha^-}} {N_\alpha}$$
 where $N_\alpha$ is the constructible model over $\Mbar_{<\alpha}$ enumerated by $\cbar_\alpha$.  From above,
 $\cbar_\alpha$ is also a construction sequence over $\Mbar_{<\alpha} M_\alpha$.  By PMOP, there is a constructible model $N_{\alpha+1}$
 over $M_\alpha N_\alpha$, from which it follows there is an enumeration $\cbar_{\alpha+1}$ of $N_{\alpha+1}$ in which $\cbar_\alpha$ is an initial segment.
 \qed\end{proof}

\begin{Lemma} \label{finitetree} ($T$ stable)
Suppose $(M_\eta:\eta\in I)$ is any independent tree of
models indexed by a finite tree $(I,\trianglelefteq)$.  Then the set
$\bigcup_{\eta\in I} M_\eta$ is essentially finite  with respect to any strong type
$p$ that is orthogonal to every $M_\eta$ (see Definition~\ref{essentiallyfinite}).
\end{Lemma}

\begin{proof} We argue by induction on $|I|$.  For $|I|=1$, this is immediate by Lemma~\ref{basicorth2}(1) (taking $A=M_{\<\>}$ and $B=\emptyset$).
So assume $(M_\eta:\eta\in I)$ is any independent tree of
models with
$|I|=n+1$ and we have proved the Lemma when $|I|=n$.
Fix any strong type $p$ that is orthogonal to every $M_\eta$.  
Choose any leaf $\eta\in I$ and let $J\subseteq I$ be the subtree with universe
$I\setminus\{\eta\}$.  By the inductive hypothesis, $M_J$
is essentially finite with respect to $p$,  so  the result follows by
Lemma~\ref{basicorth2}(2), taking $A=M_J$ and
$B=M_\eta$.
\qed\end{proof}

\begin{Lemma}  \label{anytree}  
Suppose $(M_\eta:\eta\in I)$ is any independent tree of
models indexed by any tree $(I,\trianglelefteq)$ and let $N$ be any model 
that contains and is atomic over $M_I$.
Let $p\in S(N)$ be any regular type $\perp\PP_e$ and $\perp M_\eta$ for every $\eta\in I$.
Then $Nc$ is an atomic set over
$M_I$ for every realization $c$ of $p$.
\end{Lemma}

\begin{proof} 
As notation, for $K\subseteq I$, we let $M_K$ denote $\bigcup_{\nu\in K} M_\nu$.
By Lemma~\ref{stationary} there is a finite $d_0\subseteq N$ over which $\tp(c/N)$ is based and stationary.
It suffices to show that $\tp(dc/M_I)$ is isolated for any finite $d$ with  $d_0\subseteq d\subseteq N$,
 so choose such a $d$.  Choose a finite $e\subseteq M_I$ with a formula
$\phi(x,e)$ isolating $\tp(d/M_I)$.  Choose a finite, downward closed subtree $J\subseteq I$ containing $e$.
As $\tp(c/d)$ is stationary and $\perp M_\eta$ for every $\eta\in J$, Lemma~\ref{finitetree} implies that $M_J$ is essentially finite with respect to $\tp(c/d)$,
so there is a finite $e^*$, $e\subseteq e^*\subseteq M_J$
for which $\tp(c/de^*)\vdash\tp(c/dM_J)$.   As $\tp(c/N)\perp\PP_e$,  by Lemma~\ref{AIlemma}, $\tp(c/de^*)$ is isolated.
Thus, $\tp(c/dM_J)$ is isolated as well. Since $\tp(d/M_J)$ is isolated, so is $\tp(cd/M_J)$.
But now, since $M_J\subseteq_{TV} M_I$, we conclude that $\tp(cd/M_I)$ is isolated as well.
\qed\end{proof}


We define three species  of decompositions.    We begin with the least constrained.

\begin{Definition}  \label{decompdef}
{\em  
Fix a model $M$.  A {\em weak decomposition  $\d=\<(M_\eta,a_\eta):\eta\in I\>$ {\bf inside} $M$\/} 
consists of an independent tree $\d=\{M_\eta:\eta\in I\}$ 
of countable, $na$-substructures $M_\eta\subseteq_{na} M$
indexed by $(I,\trianglelefteq)$, and a distinguished finite tuple $a_\eta\in M_\eta$ (but $a_{\<\>}$ is meaningless)
satisfying the following conditions for each $\eta\in I$:

\begin{enumerate}

\item  The set $C_\eta:=\{a_\nu:\nu\in Succ_I(\eta)\}$ is independent over $M_\eta$;
\item  For each $\nu\in Succ_I(\eta)$ we have:
\begin{enumerate}  
\item  If $\eta\neq\<\>$, then $\tp(a_\nu/M_\eta)\perp M_{\eta^-}$; 
\item  $M_\nu$ is dominated by $a_\nu$ over  $M_\eta$;
\end{enumerate}
\end{enumerate}

A {\em regular decomposition inside $M$} is a  weak decomposition inside $M$ such that   $\tp(a_\nu/M_\eta)$ is a regular type for every $\eta\in I$ and $a_\nu\in C_\eta$.
\\
A {\em $\PP_e$-decomposition inside $M$} has each $\tp(a_\nu/M_\eta)\in \PP_e$.
}
\end{Definition}

For a given $M$, let $K_{\PPe}\subseteq K_{reg}\subseteq K_{wk}$ denote the sets of [$\PPe$, regular, weak] decompositions $\d$ inside $M$.
For each of these notions, there are two ways in which a decomposition $\d$ can be maximal.  Thankfully, both notions are equivalent.

We can define a natural partial order $\le^*$ on each of $K_{\PPe}$, $K_{reg}$, $K_{wk}$ by increasing the index tree, but leaving the nodes unchanged.  That is
say
$$\d=\<M_\eta,a_\eta:\eta\in J\>\le^*\d'=\<M'_\eta,a'_\eta:\eta\in I\>$$
 if and only if  the index tree $(J,\trianglelefteq)$ is a downward closed subtree of 
$(I,\trianglelefteq)$ and $(M_\eta,a_\eta)=(M'_\eta,a'_\eta)$ for all $\eta\in J$.

\begin{Lemma}  \label{extendtomax}  Fix any model $M$ and any [weak,regular,$\PPe$] decomposition $\d=\<M_\eta,a_\eta:\eta\in I\>$ inside $M$.  Then
\begin{enumerate}
\item  $\d$ is $\le^*$-maximal inside $M$ if and only if $C_\eta$ is maximal for every $\eta\in I$; and
\item  
Every  [weak, regular, $\PPe$] decomposition $\d$ inside $M$ can be $\le^*$-extended to a maximal [weak, regular, $\PP_e$] decomposition
inside $M$.
\end{enumerate}
\end{Lemma}  

\begin{proof}  
(1)   If $\d'$ is a proper $\le^*$-extension of $\d$, then it is obvious that some $C_\eta$ gets extended.
For the converse, suppose there is some $\eta\in I$ for which $C_\eta$ can be extended.  Choose a $\trianglelefteq$-least such $\eta$ and choose $a^*\in M$ so that
$C_\eta\cup\{a^*\}$ satisfies the constraints.   Let $I^+=I\cup\{\eta^+\}$ be the one-point extension of $(I,\trianglelefteq)$ whose extra note is a leaf, with $\eta\triangleleft \eta^+$.
Since $M_\eta\subseteq_{na} M$, we can use Fact~\ref{3big}(3) to choose $M_{\eta^+}\subseteq_{na} M$ with  $M_\eta a^*\subseteq M_{\eta^+}$ and $M_{\eta^+}$ dominated by $a^*$ over $M_\eta$. 
The verification that $(M_\eta:\eta\in I^+)$ remains an independent tree of models follows from the domination and the fact that $\tp(a^*/M_\eta)\perp M_{\eta^-}$.
Then $\d^+=\d\smallfrown\<M_{\eta^+},a^*\>$ properly $\le^*$-extends $\d$.

(2)
It is evident that decompositions inside $M$ of any species are closed under $\le^*$-chains, so (2) follows by (1) and  Zorn's Lemma.
\qed\end{proof}

\begin{Definition}  {\em  
A {\em [weak, regular, $\PP_e$] decomposition {\bf of} $M$\/} is a
maximal [weak, regular, $\PP_e$] decomposition (in either of these senses).
}
\end{Definition}

\begin{Lemma}[$T$ V-DI] \label{overatomic}
  Let $M$ be any model, let $\d=\<M_\eta,a_\eta:\eta\in I\>$
be any [weak, regular, $\PP_e$]  decomposition inside $M$,  and let $N$ be atomic over $M_I$.
If $p\in\PP_e$ is any regular type with $p\not\perp N$, then $p\not\perp M_\eta$ for some $\eta\in I$.
\end{Lemma}

\begin{proof}  Recall that  V-DI implies PMOP and $\PP_e$-NDOP by Theorem~\ref{DIPMOP} and Proposition~\ref{VDINDOP}.
We first prove the Lemma for all finite index trees $(I,\trianglelefteq)$ by induction on $|I|$.
To begin, if $|I|=1$, then we must have $N=M_{\<\>}$ and there is nothing to prove.
Assume the Lemma holds for all trees of size $n$ and let
  $\d=\<M_\eta,a_\eta:\eta\in I\>$ be a decomposition inside $M$ indexed by $(I,\trianglelefteq)$ of size
$n+1$.  Let $N$ be atomic over 
$\bigcup_{\eta\in I} M_\eta$ and let $p\in\PP_e$ be  non-orthogonal to $N$.  
Choose 
a leaf $\eta\in I$ and let $J=I\setminus\{\eta\}$.  
If $(I,\trianglelefteq)$ were a linear order, then again $N=M_\eta$ and there is nothing to prove.
If $(I,\trianglelefteq)$ is not a linear order, then by Lemma~\ref{usePMOP},
choose any $N_J\preceq N$ to be constructible over
$M_J$.  
By Lemma~\ref{relaxamodel}, either $p\not\perp M_\nu$ or $p\not\perp N_J$.  In the first case we are done,
and in the second we finish by the inductive hypothesis since $|J|=n$.
Thus, we have proved the Lemma whenever the indexing tree $I$ is finite.

For the general case, fix a decomposition 
  $\d=\<M_\eta,a_\eta:\eta\in I\>$
inside $M$, let $N$ be atomic over $M_I$.   Fix any $p\in\PP_e$ with $p\not\perp N$.
Choose $q\in S(N)$, $q\not\perp p$ and choose a finite set $d\subseteq N$ on which $q$ is based.
As $N$ is atomic over $M_I$, we can find a finite subtree $J\subseteq I$ such that $\tp(d/M_J)$ is isolated.
As $M_J$ is countable, use Fact~\ref{arbitrary}(2) to choose a constructible model $N'\preceq N$ over $M_J$ with $d\subseteq N'$. 
As $d\subseteq N'$, $p\not\perp N'$, and $J$ is finite, it follows from above that  $p\not\perp M_\eta$ for some $\eta\in J$.
\qed\end{proof}


\begin{Theorem}[V-DI]   \label{Pmax}  Let $M$ be any model and let $\d=\<M_\eta,a_\eta:\eta\in I\>$ be any [weak,regular,$\PPe$] decomposition of $M$.
Then $N\preceq M$ is a dull pair for any $N\preceq M$ containing $M_I$.

\end{Theorem}

\begin{proof}   Choose any $N\preceq M$ with $M_I\subseteq N$.  By Lemma~\ref{usePMOP}, choose a constructible model $N_0\preceq N$ over $M_I$.
We first show that $N_0\preceq M$ is a dull pair.  If this were not the case, then by Theorem~\ref{chardull}, there would be some $a\in M\setminus N_0$ with
$p=\tp(a/N_0)\in \PPe$.  By Lemma~\ref{relaxamodel}, $p\not\perp M_\eta$ for some $\eta\in I$.  Choose a $\trianglelefteq$-least such $\eta\in I$.  
Since $M_\eta\subseteq_{na} M$, the 3-model Lemma gives us some $h\in M$ such that $\tp(h/M_\eta)$ is regular and $\not\perp p$ with 
$\fg h {M_\eta} {N_0}$.  
It follows that $\tp(h/M_\eta)\in\PPe$ and, 
by the minimality of $\eta$, we have that $\tp(h/M_{\eta})\perp M_{\eta^-}$ (provided $\eta\neq\<\>$).  Regardless of the species of $\d$ [weak,regular,$\PPe$]
this contradicts the maximality of $\d$.  Thus, $N_0\preceq M$ is dull.  That $N\preceq M$ is dull now follows from Corollary~\ref{splitfiltration}(1).
\qed\end{proof}

The following Corollaries follow easily.  

\begin{Corollary}[V-DI]  \label{next} Suppose $M$ be any model and let $\d=\<M_\eta,a_\eta:\eta\in I\>$ be any [weak,regular,$\PPe$] decomposition of $M$.
Then $M$ and $N$ are back-and-forth equivalent for every $N\preceq M$ containing $M_I$.
\end{Corollary}

\begin{proof}  Immediate from Theorems~\ref{Pmax} and \ref{absolute}.
\qed\end{proof}

\begin{Corollary}[V-DI]  \label{firstone}   Suppose $M$ and $N$ are models, and the same $\d=\<M_\eta,a_\eta:\eta\in I\>$ is a [weak,regular,$\PPe$] decomposition of 
both $M$ and $N$.  Then $M\equiv_{\infty,\omega} N$.
\end{Corollary}

\begin{proof}  By Lemma~\ref{usePMOP} there is a constructible model $M'\preceq M$ over $M_I$.  By replacing $N$ by a conjugate over $M_I$, we may additionally
assume that $M'\preceq N$.  Two applications of  Corollary~\ref{next} yield $M\equiv_{\infty,\omega} M'\equiv_{\infty,\omega} N$.
\qed\end{proof}

Theorem~\ref{Pmax} and Corollaries~\ref{next} and \ref{firstone}  encapsulate what can be said for $\PPe$-decompositions of $M$ as they do not touch any of the always isolated types.
However, for weak and regular decompositions of $M$, we can say considerably more.   

\begin{Theorem}[V-DI]   \label{atomicdecomp}  Let $M$ be any model and let $\d=\<M_\eta,a_\eta:\eta\in I\>$ be any [weak,regular] decomposition of $M$.
Then $M$ is atomic over $M_I$.  
\end{Theorem}

\begin{proof}  We first prove this when $M$ is countable.  
By Lemma~\ref{usePMOP}, we know there is a constructible, hence atomic, model $N_0\preceq M$ over $M_I$
so choose $N\preceq M$ to be maximal atomic over $M_I$.  We argue that $N=M$.  If this were not the case,
choose some $e\in M\setminus N$ such that $p=\tp(e/N)$ is regular.    By Theorem~\ref{Pmax},  $N\preceq M$ is dull, hence $p\perp \PPe$.
Additionally, 

\smallskip
\noindent{{\bf Claim.}}  $p\perp M_\eta$ for all $\eta\in I$.

\begin{proof}
Suppose this were not the case.
Choose $\eta\in I$ $\trianglelefteq$-minimal
 such that $p\not\perp M_\eta$.  Thus, either $\eta=\<\>$ or
$p\perp M_{\eta^-}$.  
By the 3-model Lemma, Fact~\ref{3big}(2), there is an element $h\in M$ such that $\tp(eh/M_\eta)$
is regular and non-orthogonal to $p$ (hence orthogonal to $M_{\eta^-}$ if $\eta\neq\<\>$), 
but $\fg h {M_\eta} {N_\alpha}$.
This element $h$ contradicts the maximality of $C_\eta$.
\qed\end{proof} 

By the Claim and $p\perp\PPe$, Lemma~\ref{anytree} implies $Nc$ is atomic over $M_I$.   As well, since $M$ and hence $N$ is countable with $\tp(b/N)\in \PPe$,
 a constructible model $N'\preceq M$ over $Nc$ exists by Corollary~\ref{prime}.   It follows that $N'$ is atomic over $M_I$, contradicting the maximality of $N$.

For the general case, suppose $M$ is uncountable.  Choose a forcing extension $\V[G]$ of $\V$ in which $M$ is countable.  As a counterexample to the maximality of $\d$ would be given by a finite tuple, it follows that $\d$ is a decomposition of $M$ in $\V[G]$.  Thus, by the argument above in $\V[G]$,  $M$ is atomic over $M_I$.  That is, every finite tuple from $M$
isolated by a formula over $M_I$.  The same formulas witness that $M$ is atomic over $M_I$ in $\V$ as well.  
\qed\end{proof} 

%
%
%

\begin{Corollary}[V-DI]  \label{part2}  Suppose $M$ and $N$ are two models that share the same maximal [weak, regular]-decomposition $\d=\<M_\eta,a_\eta:\eta\in I\>$.
Then $M$ and $N$ are back-and-forth equivalent over $M_I$.
\end{Corollary}

\begin{proof}  As both $M,N$ are atomic over $M_I$ by Theorem~\ref{atomicdecomp}, the set 
 $$\F=\{\hbox{partial, elementary}\ f:M\rightarrow N: \dom(f)=\abar M_I, f\mr{M_I}=id\}$$ 
 is a back-and-forth system over $M_I$.
 \qed\end{proof}

\section{Equivalents of NOTOP}

In Chapter XII of \cite{Shc}, Shelah defines a theory having OTOP.

\begin{Definition}   \label{OTOP}   {\em  Let $T$ be a countable, superstable theory.  We say {\em $T$ has OTOP} if there is a type $p(\xbar,\ybar,\zbar)$ with $\lg(\ybar)=\lg(\zbar)$ such that, 
for all infinite cardinals $\lambda$ and all binary relations $R\subseteq \lambda^2$, there is a model $M_R\models T$ and $\{\abar_\alpha:\alpha\in\lambda\}$ such that
$M_R$ realizes the type $p(\xbar,\abar_\alpha,\abar_\beta)$ if and only if $R(\alpha,\beta)$ holds.  
\\
We say {\em $T$ has NOTOP} if it fails to have OTOP.
}
\end{Definition}

Seeing this definition, one could imagine a weakening that is reminiscent of the distinction between a formula $\phi(\xbar,\ybar)$ being unstable (i.e., has the order property) and $\phi(\xbar,\ybar)$ having the Independence Property.  

\begin{Definition}  \label{linOTOP}  {\em  Let $T$ be a countable, superstable theory.  We say {\em $T$ has linear OTOP} if there is a type $p(\xbar,\ybar,\zbar)$ with $\lg(\ybar)=\lg(\zbar)$ such that, 
for all infinite cardinals $\lambda$,
there is a model $M_\lambda\models T$ and $\{\abar_\alpha:\alpha\in\lambda\}$ such that
$M_\lambda$ realizes the type $p(\xbar,\abar_\alpha,\abar_\beta)$ if and only if $\alpha\le\beta$.
\\
We say {\em $T$ has linear NOTOP}  if it fails to have linear OTOP.
}
\end{Definition}

It is somewhat curious that among superstable theories, OTOP and linear OTOP coincide, since for a first order formula $\phi(x,y)$ (as opposed to a type)
whether it codes an order is equivalent to $\phi(x,y)$ being unstable, whereas its coding arbitrary binary relations is equivalent to $\phi(x,y)$ having the Independence Property.

Here, with Theorem~\ref{finish} below, we prove the two notions are equivalent by demonstrating that each is equivalent to V-DI.   
We begin with one Lemma that is of independent interest.

\begin{Lemma}  \label{finitepiece}  Suppose $T$ is superstable, $\kappa$ any uncountable regular cardinal,  and $(\abar_\alpha:\alpha<\kappa)$ is any sequence of finite tuples.
Then there is a stationary $S\subseteq\kappa$ and a finite $F\subseteq \bigcup\{\abar_\alpha:\alpha\}$ such that $\{\abar_\alpha:\alpha\in S\}$ is independent over $F$.
\end{Lemma}

\begin{proof}  Let $S_0=\{\alpha\in\kappa:\alpha$ is a limit ordinal$\}$.  For each $\alpha\in S_0$, find some $\beta(\alpha)<\alpha$ such that $\tp(\abar_\alpha/A_\alpha)$ does not fork over $A_{\beta(\alpha)}$, where $A_\alpha:=\bigcup\{\abar_\gamma:\gamma<\alpha\}$.  By Fodor's Lemma there is some $\beta^*<\kappa$ and a stationary $S_1\subseteq S_0$
such that, taking $B=A_{\beta^*}$,  $\fg {\abar_\alpha}  B {A_\alpha}$ for all $\alpha\in S_1$.  Furthermore, for each $\alpha\in S_1$, there is a finite $F\subseteq B$ for which
$\fg {\abar_\alpha} F B$.  As $|B|<\kappa$, there is a stationary $S_2\subseteq S_1$ and some finite $F^*\subseteq B$ for which
$\fg {\abar_\alpha} {F^*} B$ holds for all $\alpha\in S_2$.  Then $\{\abar_\alpha:\alpha\in S_2\}$ is independent over $F^*$.
\qed\end{proof}

\begin{Theorem}  \label{finish}  The following are equivalent for a countable, superstable theory $T$.
\begin{enumerate}
\item  $T$ has V-DI;
\item  $T$ has $\PPe$-NDOP and PMOP;
\item  $T$ has $\PPe$-NDOP and countable PMOP;
\item  $T$ has linear NOTOP;
\item  $T$ has NOTOP.
\end{enumerate}
\end{Theorem}

\begin{proof}  The equivalence of (1)--(3) was proved in Theorem~\ref{sofar}, and $(4)\Rightarrow(5)$ is trivial.  

$(5)\Rightarrow(1)$ is proved on pages 122-124 of \cite{Hart}.  Hart's condition (*) is precisely V-DI in our notation.  Another proof of this is given in Section XII.4 of \cite{Shc}.  

So, it remains to prove $(1)\Rightarrow(4)$.  The punchline for this argument is that the existence of a tree decomposition for a purported model witnessing linear OTOP
gives too many symmetries of the structure.  We use these symmetries to show that, for $\kappa$ large enough,
if a model $M^*$ realizes $p(\xbar,\abar_\alpha,\abar_\beta)$ whenever $\alpha<\beta<\kappa$, then it must also realize $p(\xbar,\abar_\beta,\abar_\alpha)$ for some
carefully chosen $\alpha<\beta<\kappa$.  

Choose a sufficiently large regular cardinal $\kappa$, a sequence $(\abar_\alpha:\alpha<\kappa)$, and a model $M^*$ containing $(\abar_\alpha:\alpha<\kappa)$ for which  $M^*$ realizes $p(\xbar,\abar_\alpha,\abar_\beta)$ whenever $\alpha<\beta$.  We will find some $\beta<\alpha$ for which $p(\xbar,\abar_\alpha,\abar_\beta)$ is realized as well. 
For this,  by passing to  a large subsequence, Lemma~\ref{finitepiece} allows us to assume that  $\{\abar_\alpha:\alpha\in\kappa\}$ is independent over a finite set $F$.
Now choose a countable $M_{\<\>}\subseteq_{na} M^*$ containing $F$.   As $\kappa>2^{\aleph_0}$ and passing to a further subsequence, we may additionally assume
that $\tp(\abar_\alpha/M_{\<\>})=\tp(\abar_\beta/M_{\<\>})$.  
By removing at most countably many elements, we may assume that $\{\abar_\alpha:\alpha\in\kappa\}$
is independent over $M_{\<\>}$.  Next, for each $\alpha\in\kappa$, using Fact~\ref{3big}(3),  choose a countable $M_{\<\alpha\>}\subseteq_{na} M^*$ that is dominated by $\abar_\alpha$ over $M_{\<\>}$.  Thus, $\d_0=\<M_{\<\alpha\>},\abar_\alpha:\alpha<\kappa\>\smallfrown M_{\<\>}$ is a weak decomposition inside $M^*$.  By Lemma~\ref{extendtomax}(2),
there is a $\le^*$-extension $\d=\<M_\eta,\abar_\eta:\eta\in J\>$ of $\d_0$ that is a weak decomposition of $M^*$.   
Thus, by Theorem~\ref{atomicdecomp}, $M^*$ is atomic over $M_J$.  

For any $\alpha<\beta$, choose a realization $\dbar_{\alpha,\beta}$ of $p(\xbar,\abar_\alpha,\abar_\beta)$ in $M^*$.   As $\tp(\dbar_{\alpha,\beta}/M_J)$ is isolated, choose
$\ebar_{\alpha,\beta}\supseteq\abar_\alpha\abar_\beta$ from $M_J$ and a formula $\theta(\xbar,\ebar_{\alpha,\beta})$ isolating $\tp(\dbar_{\alpha,\beta}/M_J)$.
Thus, in particular, $\theta(\xbar,\ebar_{\alpha,\beta})$ is consistent, and every realization of it realizes $p(\xbar,\abar_\alpha,\abar_\beta)$.

As notation, for each $\gamma$, let $M_J(\gamma)=\bigcup\{M_\eta:\eta\trianglerighteq\<\gamma\>\}\setminus M_{\<\>}$.
Each pair $\alpha<\beta$ induces a partition of $M_J$ into three (disjoint) pieces, namely $M_J(\alpha)$, $M_J(\beta)$, and the complement, 
$M_J\setminus (M_J(\alpha)\cup M_J(\beta))$.  This partition induces a partition of 
$\ebar_{\alpha,\beta}=\rbar_{\alpha,\beta}\sbar_{\alpha,\beta}\tbar_{\alpha,\beta}$.
Crucially, note that for all $\alpha<\beta$ the tuples $\{\rbar_{\alpha,\beta},\sbar_{\alpha,\beta},\tbar_{\alpha,\beta}\}$ are independent over $M_{\<\>}$.

By Erd\"os-Rado, there is a large subsequence $I\subseteq \kappa$ such that, for all $\alpha<\beta$ from $I$,
\begin{itemize}
\item  The $L$-formula $\theta(\xbar,\ybar)$  for which $\theta(\xbar,\ebar_{\alpha,\beta})$ isolates $\tp(\dbar_{\alpha,\beta}/M_I)$ is constant; 
\item  The partition of $\ebar_{\alpha,\beta}=\rbar_{\alpha,\beta}\sbar_{\alpha,\beta}\tbar_{\alpha,\beta}$ is independent of $\alpha<\beta$;
\item  There are unique types $r^*,s^*,t^*\in S(M_{\<\>})$ with 
$\tp(\rbar_{\alpha,\beta}/M_{\<\>})=r^*$,
$\tp(\sbar_{\alpha,\beta}/M_{\<\>})=s^*$, and $\tp(\tbar_{\alpha,\beta}/M_{\<\>})=t^*$.
\end{itemize}

Let $I_0\subseteq I$ have the maximum and minimum elements of $I$ removed (if they exist).    Thus, for every $\gamma\in I_0$,
$M_J(\gamma)$ contains a realization of $r^*$ (in particular, $\rbar_{\gamma,\beta}$ for any $\beta\in I$, $\beta> \gamma$), and dually
$M_J(\gamma)$ contains a realization of $s^*$.

Now choose $\beta<\alpha$ from $I_0$.  From above, choose $\shat\in M_J(\beta)$ realizing $s^*$,  $\rhat\in M_J(\alpha)$ and let $\that=\tbar_{\beta.\alpha}$.
Put $\ehat:=\rhat\shat\that$.
Then, as $\{\rhat,\shat,\that\}$ are independent over $M_{\<\>}$, we have $\tp(\ehat/M_{\<\>})=\tp(\ebar_{\beta,\alpha}/M_{\<\>})$.  It follows
that $M^*\models \exists\xbar\theta(\xbar,\ehat)$ and, moreover, any such realization of $\theta(\xbar,\ehat)$ realizes $p(\xbar,\abar_\alpha,\abar_\beta)$.
\qed\end{proof}

\section{Some context for $\PPe$}  Many of the results presented here generalize results of Shelah and the first author \cite{Borel} under the stronger assumption of $\omega$-stability.
Recall that a regular type $p$ is {\em eventually non-isolated,}  eni,
if there is a finite set $b$ on which it is based and stationary, and a model $M\supseteq b$ that omits the restriction $p|b$.
In \cite{Borel}, 
it is proved that for $T$ $\omega$-stable, NOTOP is equivalent to eni-NDOP.  Noting that because an $\omega$-stable has constructible models over every subset,
the following Corollary is immediate from Theorem~\ref{finish}, since both properties are equivalent to NOTOP.

\begin{Corollary}  \label{eni}  If $T$ is $\omega$-stable, then $\PPe$-NDOP is equivalent to eni-NDOP.
\end{Corollary}

However, outside of $\omega$-stable theories, the following example shows that the notions are distinct.

\begin{Example} \label{Ex}  A small, superstable theory $T$ with $\PPe$-DOP, but eni-NDOP.
\end{Example}

  Let $L=\{U,+,0,U_n\}_{n\in\omega} \cup\{P, E_1,E_2\}\cup\{V,\pi,g,V_n\}_{n\in\omega}$ and  fix a cardinal $\kappa\ge 2^{\aleph_0}$.  We will describe a saturated $L$-structure
$M$ of size $\kappa$, and $T=Th(M)$ will be as claimed.  

The two unary predicates $U(M)$ and $V(M)$ partition the universe of $M$.  One sort, $(U(M),+,0,U_n)_{n\in\omega}$ is a
($\kappa$-dimensional) ${\bf F}_2$-vector space with a nested sequence of subspaces $U_n(M)$, where $U_0(M)=U(M)$ and each $U_{n+1}(M)$ has co-dimension one in $U_n(M)$. 
As $M$ is saturated, $\bigcap_{n\in\omega} U_n(M)$  has dimension $\kappa$.  

The unary predicate $P(M)\subseteq \bigcap_{n\in\omega} U_n(M)$ consists of a linearly independent set (of size $\kappa$) such that
its linear span $\<P(M)\>$ has co-dimension $\kappa$ in $\bigcap_{n\in\omega} U_n$.  

So far, $P(M)$ has no structure, i.e., is totally indiscernible.  However, the binary relations $E_1,E_2$ are interpreted as cross-cutting equivalence relations, each with
infinitely many classes,  on
$P(M)$, with $E_i(M)\subseteq P(M)\times P(M)$ for each $i$.  As notation, let $E^*(x,y):=E_1(x,y)\wedge E_2(x,y)$.   As $M$ is saturated, $E^*(M,a)\subseteq P(M)$ has size 
$\kappa$ for every $a\in P(M)$.  One should think of $(P(M),E_1,E_2,E^*)$ as coding the `standard DOP checkerboard.'

Continuing, $V(M)$ has size $\kappa$, and $\pi:V(M)\rightarrow P(M)$ is a surjection.  
The function $g:U\times V\rightarrow V$ is a group action that acts regularly on $\pi^{-1}(a)$ for every $a\in P(M)$.

Finally, for each $n\in\omega$, $V_n$ is interpreted so that $V_{n+1} \subseteq V_n$ for each $n$, and for every $a\in P(M)$, the set  $\{u\in U(M):g(u,v)\in V_n(M)\cap \pi^{-1}(a)\}$ is a coset of $U_n(M)$. 


It is readily checked that $T=Th(M)$ is small and superstable.   We describe the five species of regular types occurring in $S(M)$.  Let $p_0$ be the complete type
asserting that $P(x)$ holds, but $\neg E_1(x,a)\wedge \neg E_2(x,a)$ for every $a\in P(M)$.  For $i=1,2$, let $p_i(x,a)$ assert that $E_i(x,a)$ holds, but
$E_{3-i}(x,b)$ fails for every $b\in P(M)$.
Let  $p^*(x,a)$ be generated by $E^*(x,a)\wedge x\neq a$, and let $q(x)$ be the complete type asserting that $x\in \bigcap U_n(M)\setminus \<P(M)\>$.
As $\dcl(c)\cap U(M)\neq\emptyset$ for any finite tuple $c$, every non-algebraic type is non-orthogonal to one of these types.  
As it is akin to the DOP checkerboard, $p^*(x,a)$ has a DOP witness, and 
it can be checked that the only regular types with a DOP witness are those non-orthogonal to $p^*(x,a)$. 

We argue that $p^*(x,a)\in\PPe$.   Choose any $b\in P(M)$ with $E^*(b,a)\wedge b\neq a$ and choose any $c\in \pi^{-1}(b)\cap\bigcap V_n$.
The type $w(x,a):=\tp(c/a)$ is visibly non-isolated, and since $bc$ is dominated by $b$ over $a$, $w(x,a)$ has weight one.  As $b\in\dcl(c)$ realizes $p^*(x,a)$,
$w(x,a)\not\perp p^*(x,a)$ so $p^*(x,a)\in \PPe$.  Thus, $T$ has $\PPe$-DOP.  

The type $w(x,a)$ is not regular, and 
 it can be shown that the non-orthogonality class of $p^*(x,a)$ does not contain any eni (regular) type.  Thus, $T$ has eni-NDOP.

\medskip

We note two extreme cases of $\PPe$.

\begin{Corollary}  \label{noPe}  Assume $T$ is countable and superstable, but $\PPe=\emptyset$.  Then $T$ is $\omega$-stable and $\omega$-categorical.
\end{Corollary}

\begin{proof}  Let $M$ be any countable model of $T$.  If $\PP_e=\emptyset$, then by Lemma~\ref{stationary}, every $p\in S(M)$ is based and stationary over a finite set.
As $M$ is countable, this implies $S(M)$ is countable, hence $T$ is $\omega$-stable.   
To get that $T$ is $\omega$-categorical, we argue that every countable model is saturated.  Since $T$ is $\omega$-stable, a countable, saturated model $N$ exists.
Let $M$ be any countable model of $T$.  Since $N$ is countably universal, we may assume $M\preceq N$.  Since $\PPe=\emptyset$, it is also dull.
Thus, by Proposition~\ref{iso}, $M$ is isomorphic to $N$, so $M$ is saturated as well.
\qed\end{proof}  

\begin{Corollary}  Assume $T$ is countable and superstable, but every regular type is in $\PPe$.  Then $T$ has NOTOP if and only if $T$ is classifiable.
\end{Corollary}

\begin{proof} Right to left is immediate.   By our assumption on $T$,  $\PPe$-NDOP is equivalent to NDOP, so left to right follows from Theorem~\ref{finish}.  
\qed\end{proof}

\begin{Remark}
{\em We close by observing that the countability of the language is crucial in the equivalents of Theorem~\ref{finish}.  Indeed, the notion of V-DI is preserved under the addition or deletion of constant symbols, but NOTOP is not.  In particular, the theory in Example~\ref{Ex} has OTOP, since the language is countable and $T$ has $\PPe$-DOP.  However, if one expands by adding constants for an a-model (equivalently, replacing the theory $T$ by the elementary diagram of $M$) then the expanded theory cannot code arbitrary relations, hence has NOTOP.  
}
\end{Remark}

\appendix

\section{Appendix}  

\subsection{Isolation, construction sequences, and TV-substructures}

The initial definitions are well known.

\begin{Definition}  {\em  An $n$-type $p(x)\in S_n(A)$ is {\em isolated} if $\psi(x,a)\vdash p(x)$ for some $\psi(x,a\in p$.\\
A model $N$ is {\em atomic over $B$} if $B\subseteq N$ and $\tp(a/B)$ is isolated
for every finite tuple $a$ from $N$.\\
A {\em construction sequence over $B$} is a sequence $\cbar=(a_\alpha:\alpha<\delta)$ with, for each $\alpha<\delta$,
 $\tp(a_{\alpha}/BA_\alpha)$ isolated, where $A_\alpha=\bigcup\{a_\beta:\beta<\alpha\}$.\\
 A model $N$ is {\em constructible over $B$} if there is a construction sequence $\cbar$ over $B$ and the universe of $N=B\cup\bigcup\cbar$.
 }
 \end{Definition}
 
 The following facts are also well known.
 
 \begin{Fact}  \label{arbitrary} Let $T$ be any complete theory.
 \begin{enumerate}
 \item  If $N$ is constructible over $B$, then $N$ is atomic over $B$.
 \item  If $N$ is atomic over $B$ and is countable, then any enumeration of $N$ of order type $\omega$ is a construction sequence over $B$.
 \item  If $B\subseteq M$ and $\cbar$ is a construction sequence over $B$, then there is $\cbar'\subseteq M$ with $\tp(\cbar/B)=\tp(\cbar'/B)$.
 \end{enumerate}
 \end{Fact}

 In many places, we use the Open Mapping Theorem, which holds for an arbitrary stable theory.
 
 \begin{Fact}[Open Mapping Theorem]  \label{Open}  If $A\subseteq B$ and $\fg c A B$, then $\tp(c/B)$ isolated implies $\tp(c/A)$ isolated.
 \end{Fact}
 
 \begin{Definition}  \label{TVdef}  {\em  For any sets $A,B$, we say {\em $A$ is a Tarski-Vaught subset of $B$}, written $A\subseteq_{TV} B$,\footnote{In \cite{Shc}, Shelah denotes this same notion by $\subseteq_t$}.  
  if $A\subseteq B$ and, for every $A$-definable formula  $\phi(x,a)$, 
 if there is some $b\in B$ with $\phi(b,a)$, then there is some $a^*\in A$ with $\phi(a^*,a)$.
 }
 \end{Definition}

 Obviously, for any model $M$, $M\subseteq_{na} B$ whenever $M\subseteq B$.   As well, if $\fg A M B$, then the fact that
 $MA\subseteq_{TV} MAB$ is a restatement of the Finite Satisfiability Theorem.   More generally, for arbitrary stable systems $\Mbar\sq\Nbar$ of models
 we have $\bigcup\Mbar\subseteq_{TV} \bigcup\Nbar$.
 
 Tarski-Vaught subsets play well with isolation.
 
 \begin{Lemma}  \label{TVup}
 Suppose $A\subseteq_{TV} B$ 
 and $\tp(c/A)$ is isolated.  Then $\tp(c/B)$ is isolated by the same formula and, moreover,
 $Ac\subseteq_{TV} Bc$.    Consequently, if $(\cbar_\alpha:\alpha<\gamma)$ is any construction sequence over $A$,
 then it is a construction sequence over $B$ via the same formulas.
 \end{Lemma}
 
 \begin{proof}  Suppose $\phi(x,a)$ isolates $\tp(c/A)$.   If it were not the case that $\phi(x,a)$ isolates $\tp(c/B)$, then there would
 be some $\delta(x,a,b)$ with $b$ from $B$ such that
 $$\eta(b,a):=\exists x\exists x'[\phi(x,a)\wedge\phi(x',a)\wedge\delta(x,a,b)\wedge\neg\delta(x',a,b)]$$
 However, if there were any $a^*$ from $A$ such that $\eta(a^*,a)$, this would contradict $\phi(x,a)$ isolating $\tp(c/A)$.
 
 For the moreover clause, assume $\delta(x,a',c)$ has a solution in $B$ with $a'$ from $A$.   Then 
 $$\theta(x,a,a'):=\exists z(\phi(x,a)\wedge\delta(x,a',z))$$
 also has a solution in $B$, hence in $A$.  By the isolation, any realization of $\theta(x,a,a')$ also realizes $\delta(x,a',c)$.
 The third sentence follows by induction on the length of the construction sequence.
 \qed\end{proof}
 
 \subsection{$\ell$-isolation and $\ell$-construction sequences}
 
$\ell$-isolation is a weakening of isolation.

\begin{Definition}  \label{elldef}   {\em  
A type  $\tp(c/A)$ is {\em $\ell$-isolated over $A$} if, 
for every $\phi(x,y)$, there is a formula $\psi(x,a)\in\tp(c/A)$
with $\psi(x,a)\vdash\tp_\phi(c/A)$.  
\\
A model $N$ is {\em $\ell$-atomic over $B$} if $B\subseteq N$ and $\tp(a/B)$ is $\ell$-isolated
for every finite tuple $a$ from $N$.\\
An {\em $\ell$-construction sequence over $B$} is a sequence $\cbar=(a_\alpha:\alpha<\delta)$ with, for each $\alpha<\delta$,
 $\tp(a_{\alpha}/BA_\alpha)$ $\ell$-isolated, where $A_\alpha=\bigcup\{a_\beta:\beta<\alpha\}$.\\
 A model $N$ is {\em $\ell$-constructible over $B$} if there is an $\ell$-construction sequence $\cbar$ over $B$ and the universe of $N=B\cup\bigcup\cbar$.

}
\end{Definition}

The advantage is that in a countable, superstable theory $T$, the $\ell$-isolated types are dense over any base set $A$.

\begin{Fact}  \label{ellexist}  Let $T$ be any countable, superstable theory. 
\begin{enumerate}
\item  For any base set $A$, $\{p\in S(A):p$ is $\ell$-isolated$\}$ is dense.
\item  For any set $A$ an $\ell$-constructible model  over $A$ exists.
\end{enumerate}
\end{Fact}

\begin{proof}  (1) holds by e.g., Lemma~4.2.18(4) of \cite{Shc}, and
(2) follows by iterating (1) over larger and larger approximations to a model.
\qed\end{proof}

The analogue of Lemma~\ref{TVup} holds as well, essentially by the same proof as there.

\begin{Fact}  \label{TVupl}  If $A\subseteq_{TV} B$ and $\tp(c/A)$ is $\ell$-isolated, then $\tp(c/B)$ is also $\ell$-isolated, with the same witnessing formulas.
\end{Fact}

\subsection{Orthogonality and domination}  Throughout this subsection, all that is needed is for $T$ to be stable.  The following notions are all due to Shelah.

\begin{Definition}  {\em  Suppose $p\in S(A)$ and $q\in S(B)$.  Then $p$ and $q$ are {\em orthogonal}, $p\perp q$, if,
for every $E\supseteq A\cup B$, $\fg a E b$ for every $a,b$ realizing any non-forking extensions of $p,q$, respectively.
\\  If $p\in S(D)$, we say {\em $p$ is orthogonal to the set $B$}, $p\perp B$, if $p\perp q$ for every $q\in S(B)$.
\\  If $p\in S(D)$ and $D_0\subseteq D$, we say {\em $p$ is almost orthogonal to $D_0$} if $\fg c {D} {e}$ for every $e$ such that $\fg e {D_0} D$.
} 
\end{Definition}
In many texts, being almost orthogonal is written in terms of {\em domination}. 

\begin{Definition}  \label{domdef}  {\em  We say that for $D\supseteq D_0$, {\em $cD$ is dominated by $D$ over $D_0$} 
if $\fg c {D} e$ for every
$e$ satisfying $\fg e {D_0} D$.  Thus, $cD$ is dominated by $D$ over $D_0$ if and only if $\tp(c/D)$ is almost orthogonal to $D_0$.
For arbitrary strong types $p,q$, we write $p\triangleleft q$ if, for some/every a-model $M$ on which both $p,q$ are based and stationary, for some/every $b$ realizing $q$,
there is $a$ realizing $p$ with $ab$ dominated by $b$ over $M$.
} 
\end{Definition}

We note the following two facts.

\begin{Fact}  \label{basicorth}  Suppose $\tp(e/D)$ is stationary, $D_0\subseteq D$, and $\fg D {D_0} Y$.  
\begin{enumerate}
\item  If $eD$ is dominated by $D$ over $D_0$, then $\tp(e/D)\vdash \tp(e/DY)$; and
\item  If $\tp(e/D)\perp D_0$, then $\tp(e/D)\vdash \tp(e/DY)$ and $\tp(e/D)\perp D_0Y$.
\end{enumerate}
\end{Fact}  

\begin{proof}  (1)  Choose any $e'$ such that $\tp(e'/D)=\tp(e/D)$.  Then $e'D$ is dominated by $D$ over $D_0$ as well, so
we have $\fg {e} {D} Y$ and $\fg {e'} D Y$. Thus, $\tp(e/DY)=\tp(e'/DY)$ since $\tp(e/D)$ is stationary.
(2)  The first clause follows from (1) and the second is X.1.1 of \cite{Shc}.
\qed\end{proof}

For lack of a better place, we will require the following technical lemma, whose proof is similar to the proof of $(c)\Rightarrow(d)$ of X, 2.2 of \cite{Shc}.

\begin{Lemma}  \label{adde}  Suppose $D$ is $V$-dominated by $\Abar=(A_0,A_1,A_2)$ and $\tp(e/DA_1A_2)$ is stationary and orthogonal to both $A_1$ and $A_2$.
Then $De$ is also $V$-dominated by $\Abar$.
\end{Lemma}

\begin{proof}  Choose any $\Bbar\sqsupseteq\Abar$ and we will prove in three steps that
\begin{equation} 
\tp(e/DA_1A_2)\vdash\tp(e/DA_1A_2B_0)\vdash \tp(e/DB_1A_2)\vdash \tp(e/DB_1B_2) \tag{*}
\end{equation}
which implies $\fg e {DA_1A_2} {B_1B_2}$.  Coupling this with $\fg D {A_1A_2} {B_1B_2}$ from the $V$-domination of $D$ over $\Abar$ gives
$\fg {De} {A_1A_2}{B_1B_2}$, as required.  Along the way, we  also  prove that $\tp(e/D)$ is orthogonal to both $A_1B_0$ and $A_2B_0$.

To obtain the first implication of (*), $\Abar\sq\Bbar$ and $D$ $V$-dominated by $\Abar$ give $\fg {B_0} {A_1} {A_2}$ and $\fg {D} {A_1} {A_2}$, from
which it follows that $\fg {DA_1A_2} {A_1} {B_0}$.   Thus, applying Fact~\ref{basicorth}(2) with $D_0=A_1$ gives $\tp(e/DA_1A_2)\vdash\tp(e/DA_2A_2B_0)$ and also
$\tp(e/DA_1A_2B_0)\perp A_1B_0$.  Also, note that our assumptions are symmetric between $A_1$ and $A_2$, so arguing symmetrically gives
$\tp(e/DA_1A_2B_0)\perp A_2B_0$ as well.

For the second implication, again from $\Abar\sq\Bbar$ and $D$ $V$-dominated by $\Abar$, we have $\fg {B_1}{A_1B_0} {A_2}$ and
$\fg D {A_1A_2}{B_0B_1}$.    By transitivity of non-forking we have $\fg {B_1} {A_1B_0} {DA_2}$, hence also $\fg {B_1} {A_1B_0}{DA_1A_2}$.
So, applying  Fact ~\ref{basicorth}(2) with $D_0=A_1B_0$  gives the second implication of (*).

Finally, since $\fg {B_2} {B_0A_2} {B_1}$ and $\fg {D} {A_1A_2} {B_1B_2}$, we obtain $\fg {DB_1A_2} {B_0A_2} {B_2}$.  
As we proved $\tp(e/DB_1A_2)\perp A_2B_0$ above,
applying Fact~\ref{basicorth}(2)  with $D_0=A_2B_0$ gives the third implication of (*), completing the proof of the Lemma.
\qed\end{proof}

The following notion and subsequent lemma appear as Defintion~1.4 and Lemma~1.5 of \cite{Borel}.

\begin{Definition}  \label{essentiallyfinite} {\em A set $A$ is {\em essentially finite
with respect to a strong type $p$\/} if, for all finite sets $D$ on which $p$ is based
and stationary, there is a finite $A_0\subseteq A$ such that $p|DA_0\vdash p|DA$.
}
\end{Definition}

\begin{Lemma}  \label{basicorth2}
Fix a strong type $p$.
If either of the following conditions hold
\begin{enumerate}
\item  $p\perp A$ and $B$ is a (possibly empty) $A$-independent set of finite sets; or
\item  if $A$ is essentially finite with respect to $p$, $p\perp B$, and $\fg A {A\cap B} {B}$
\end{enumerate}
then $A\cup B$ is essentially finite with respect to $p$.
\end{Lemma}

\subsection{On na-substructures}

One of the tools that led to the strong structure theorems proved for classifiable theories was the notion of an $na$-substructure, which is 
due to Hrushovski.

\begin{Definition}  \label{naDef}  {\em  We say {\em $M$ is an $na$-substructure of $N$,} written $M\subseteq_{na} N$, if $M\preceq N$ and, 
for every finite $F\subseteq M$ and every $F$-definable
formula $\phi(x)$, if there is some $c\in \phi(N)\setminus M$, then there is $c'\in\phi(M)\setminus \acl(F)$.
}
\end{Definition}

The salient features of this notion is that for any model $N$ and any countable subset $A\subseteq N$, there is a countable $M\subseteq_{na} N$ containing $A$.  This is proved by the same method as the Downward L\"owenheim-Skolem theorem.  

The following three Facts explain the utility of this notion.  Fact~\ref{3big}(1,2) appear as Propositions~8.3.5 and 8.3.6 of \cite{Pillay}, respectively, and 
Fact~\ref{3big}(3)  is Proposition~5.1~ of \cite{ShBue}.  

\begin{Fact}  \label{3big}
\begin{enumerate}
\item  Whenever $M\subseteq_{na} N$, if $p$ is any regular type non-orthogonal to $\tp(N/M)$, then some regular type $q\in S(M)$ non-orthogonal to $p$ is realized in $N\setminus M$.
\item (3-model Lemma) Suppose $M\preceq M'\preceq N$ with $M\subseteq_{na} N$.  For every regular type $p=\tp(e/M')$ with $e$ from $N$, there is $h\in N$ such that
$\tp(h/M)$ is regular, non-orthogonal to $p$, and $\fg h M {M'}$.
\item  Suppose $M\subseteq_{na} N$ and $A$ is any set such that $M\subseteq A\subseteq N$.  Then there is a model $M^*\subseteq_{na} N$ with $A\subseteq M^*$, $|M^*|=|A|$,
and
$M^*$ dominated by $A$ over $M$.
\end{enumerate}
\end{Fact}

For our purposes, we need to localize this notion of `being na' to individual regular types.
We begin with two very general lemmas.

\begin{Lemma}   \label{isolnonalg}   Suppose $M\preceq N$, $c$ from $M$ and $b\in N\setminus M$.   If $\tp(b/c)$ is isolated by $\delta(x,c)$, then there
is $b^*\in \delta(M,c)\setminus\acl(c)$.
\end{Lemma}

\begin{proof}  Since $\delta(b,c)$ holds with $b\not\in M$, $\delta(x,c)$ is non-algebraic.   As $M\preceq N$ there is $b^*\in M$ realizing $\delta(x,c)$.
We argue that any such $b^*\not\in\acl(c)$.  Suppose $b^*\in\acl(c)$.  Choose an algebraic $\alpha(x,c)\in \tp(b^*/c)$.  Since $\tp(b^*/c)$ is isolated by $\delta(x,c)$,
this would imply that $\forall x[\delta(x,c)\rightarrow\alpha(x,c)]$, but this is contradicted by $b$.
\qed\end{proof} 

 The following is a slight strengthening on the fact that in a superstable theory, a realization of a regular type can be found inside any pair of models.  
 Indeed, the proof below is simply a minor variant of Proposition~8.3.2 in \cite{Pillay}.

\begin{Lemma}   \label{aclregular} Suppose $M\preceq N$, $d$ from $M$,  and $\phi(N,d)\setminus M$ is non-empty.
Then there is some $e\in \acl(\phi(N,d)\cup M)$ with $\tp(e/M)$ regular.
\end{Lemma}

\begin{proof}    Let $D:=\acl(\phi(N,d)\cup M)$.  As in the proof of 8.3.2 of \cite{Pillay}, by Lemma 8.1.12(iii) there, since $D\not\subseteq M$, 
choose a regular type $p$ that is non-orthogonal to $\tp(D/M)$ with $R^\infty(p)=\alpha>0$
such that $\tp(D/M)$ is foreign to ${\bf R^\infty\!<\!\alpha}$.  Choose $a\subseteq D$ such that $\tp(a/M)\not\perp p$.  Note that $\tp(a/M)$ is orthogonal to all forking extensions of $p$.  There are now two cases.  If $p$ is trivial, then by Lemma~8.3.1 of \cite{Pillay}, there is $a'\in\acl(aM)\subseteq D$ such that $\tp(a'/M)$ is regular, so we are done.
On the other hand, if $p$ is non-trivial, look at the proof of Lemma~8.2.20 in \cite{Pillay}. The first two moves are to apply 7.1.17 to obtain $a_1\in\acl(aM)$ that is $p$-simple of positive $p$-weight, and then to apply 8.2.17 to find $a_2\in\dcl(Ma_1)$ such that $\tp(a_2/M)$ contains a formula $\theta$ as in the statement of 8.2.20.  Note that $a_2\in D$.
Thus, by shrinking $\theta$ slightly (but staying within $\tp(a_2/M)$) we may assume that $\theta(N)\subseteq D$.  Now, continuing with the proof of 8.2.20, find $c\in\theta(N)\subseteq D$ as there.  The verification that $\tp(c/M)$ is regular follows as in the proof of Proposition~8.3.2.
\qed\end{proof}

\begin{Definition}  {\em   For $M$ any model, a type $p\in S(M)$ is {\em na} if, for every $\phi(x,d)\in p$, there is $b\in \phi(M,d)\setminus \acl(d)$.
}
\end{Definition}  

%

\begin{Proposition}   \label{charna}  For $M\preceq N$, $M\subseteq_{na} N$ if and only if every regular $p\in S(M)$ realized in $N$ is na.
\end{Proposition}

\begin{proof}  Left  to right is obvious.
For the converse, assume $M$ is not an $na$-substructure.   Choose $\phi(x,d)$ realized in $N\setminus M$ with $d\subseteq M$, but $\phi(M,d)\subseteq \acl(d)$.
By Lemma~\ref{aclregular}, choose $e\in \acl(\phi(N,d)\cup M)$ such that $q:=\tp(e/M)$ is regular.  We show that $q$ is not $na$.
For this, choose $\abar$ from $\phi(N,d)$ for which $e\in\acl(M\abar)$.  
Say $\psi(x,\abar,d')\in \tp(e/M\abar)$ has exactly $n$ solutions and $d'\subseteq M$.
Put
$$\theta(y,d,d'):=\exists \xbar[\wedge_i \phi(x_i,d)\wedge \psi(y,\xbar,d')\wedge\exists^{=n} z \psi(z,\xbar,d')]$$
We claim that the formula $\theta(y,d,d')$ witnesses that $q$ is not na.
Clearly, $\theta(y,d,d')\in\tp(e/d,d')$, so it suffices to show  that  $\theta(M,d,d')\subseteq\acl(dd')$.
To verify this, choose $b'\in\theta(M,d,d')$ and choose $\abar'$ from $M$ witnessing this.  
Then $\abar'\subseteq\phi(M,d)$, hence $\abar\subseteq\acl(dd')$ by our choice of $\phi(x,d)$.  
But $b'\in \acl(\abar'd')$ via $\psi(z,\abar',d')$, so $b'\in\acl(dd')$, as required.  
\qed\end{proof}

\end{document}